\documentclass[11pt]{article}
\usepackage{amscd}
\usepackage{amsfonts}
\usepackage{amsmath}
\usepackage{amssymb}
\usepackage{amsthm}
\usepackage{bbm}
\usepackage{CJK}
\usepackage{fancyhdr}
\usepackage{graphicx}
\usepackage{indentfirst}
\usepackage{latexsym,bm}
\usepackage{mathrsfs}
\usepackage{xypic}
\usepackage[square, comma, sort&compress, numbers]{natbib}
\usepackage[colorlinks,linkcolor=red, anchorcolor=blue]{hyperref}
\allowdisplaybreaks[4]
\newtheorem{theorem}{Theorem}[section]
\newtheorem{lemma}[theorem]{Lemma}
\newtheorem{definition}[theorem]{Definition}
\newtheorem{proposition}[theorem]{Proposition}
\newtheorem{example}[theorem]{Example}

\newtheorem{corollary}[theorem]{Corollary}
\newtheorem{remark}[theorem]{Remark}

\usepackage[top=1in,bottom=1in,left=1.25in,right=1.25in]{geometry}
\def\<{\langle}
\def\>{\rangle}

\date{}
\begin{document}
\renewcommand{\baselinestretch}{1.2}
\renewcommand{\arraystretch}{1.0}
\title{\bf $Q$-graded Hopf quasigroups}
 \date{}
\author {{\bf Guodong Shi \quad  Shuanhong Wang\footnote {Corresponding author:  shuanhwang@seu.edu.cn}}\\
{\small School of Mathematics, Southeast University}\\
{\small Nanjing, Jiangsu 210096, P. R. of China}}
 \maketitle
\begin{center}
\begin{minipage}{12.cm}

\noindent{\bf Abstract.} Firstly, we introduce a class of new algebraic systems which generalize Hopf quasigroups and Hopf $\pi-$algebras called $Q$-graded Hopf quasigroups, and research some properties of them. Secondly, we define the representations of $Q$-graded Hopf quasigroups, i.e $Q$-graded Hopf quasimodules, research the construction method and fundamental theorem of them. Thirdly, we research the smash products of $Q$-graded Hopf quasigroups.
\\

\\

 \end{minipage}
 \end{center}
 \normalsize\vskip1cm

\section*{Introduction}
Hopf quasigroups\cite{KM1}were introduced by Klim and Majid, there are generalization of Hopf algebras\cite{Sw} that are not required to be associative. In publications we can find some results on Hopf quasigroup : Brezi\'{n}ski research the Hopf modules and the fundamental theorem\cite{Br} on Hopf quasigroups, Brezi\'{n}ski with Jiao show the smash (co)products\cite{BJ1} and $R-$smash products\cite{BJ2} on Hopf quasigoups, Jiao and Wang give the smash biproducts\cite{JW} on  quasigroups, Klim and Majid introduce the bicrossproduct\cite{KM2} on Hopf quasigroups.

On the other hand, Hopf $\pi$-coalgebras\cite{Vi} were introduced by Virelizier, and then continued by Zunino and Wang. It turms out that many of the classical results in Hopf algebras can be generalized to the Hopf $\pi$-coalgebras, Virelizier gives a generalized version of the Fundamental theorem for Hopf algebras and introduced $\pi-$integrals, Zunino introduced Yetter-Drinfeld modules\cite{Zu1}, the Drifeld double, and a generalization of the center constrction of a monoidal category\cite{Zu}, Wang introduces Doi-Hopf modules\cite{Wa2}, entwined modules\cite{Wa5}, Drinfeld'codouble\cite{Wa3} and coalgebra Galois theory\cite{Wa4} for Hopf $\pi-$coalgebras, and he proves a version of Maschke'theorem\cite{Wa1}. Hopf $\pi$-algebras are dual of Hopf $\pi$-coalgebras\cite{Tu}.

In this paper, we introduce $Q-$graded Hopf quasigroups , they generalize the Hopf quasigroups like Hopf $\pi-$algebras generalize Hopf algebras and they generalize Hopf $\pi-$algebras like Hopf quasigroups generalize Hopf algebras. The paper is organized as follows: In section 1, we recall the definitions of quasigroups, Hopf $\pi-$algebras and Hopf quasigroups. In section 2, we give the definition of $Q-$graded Hopf quasigroups and generalize many classical theorems of Hopf quasigroups to $Q-$graded Hopf quasigroups. In section 3, we define $Q-$graded Hopf quasimodules and show a construct method and fundamental theorem of Hopf module on it. In section 4, we show the smash products on $Q-$graded Hopf quasigroups.

Throughout this article, $\pi$ is a group, $Q$ is a quasigroup and all the vector spaces, tensor products and maps are over a fixed field $k$.  For a coalgebra $C$, we will use the Heyneman-Sweedler's notation $\Delta(c)=  c_{1}\otimes c_{2},$
for any $c\in C$ (summation omitted).

\section{Preliminaries}
\def\theequation{1.\arabic{equation}}
\setcounter{equation} {0}

\begin{definition}
A quasigroup is a set $Q$ with a product, identity $e$ and with the property that for each $p\in Q$ there is $p^{-1}\in Q$ satisfying
$$p^{-1}(pq)=q,~~(qp)p^{-1}=q,~~\forall q\in Q.$$

\end{definition}

It is easy to see that in any quasigroup $Q$, one has unique inverses and
$$(p^{-1})^{-1}=p,~~(pq)^{-1}=q^{-1}p^{-1},~~\forall p,q\in Q.$$

A quasigroup is $flexile$ if $p(qp)=(pq)p$ for all $p,q\in Q$ and $alternative$ if also $p(pq)=(pp)q,~p(qq)=(pq)q$ for all $p,q\in Q$. It is called $Moufang$ ir $p(q(pr))=((pq)p)r$ for all $p,q,r\in Q $ or $commutative$ if $pq=qp$ for all $p,q\in Q$ or $associative$ if $p(qr)=(pq)r,$ for all $p,q,r\in Q $. Obviously, associative quasigroup is group.

\begin{remark}
Let $Q$ be a quasigroup, then the following identities are equivalent, $\forall p,q,r\in Q,$
(1) $p(q(pr))=((pq)p)r,$
(2) $((pq)r)q=p(q(rq)),$
(3) $(pq)(rp)=(p(qr))p.$
\end{remark}

\begin{definition}
A Hopf quasigroup is a possibly-noassociative but unital algebra $H$ equipped with algebra homomorphisms $\Delta:H\rightarrow H\otimes H,~\epsilon:H\rightarrow k$ forming a coassociative coaglebra and a map $S:H\rightarrow H$ such that

$S(h_1)(h_2g)=\epsilon(h)g=h_1(S(h_2)g)$, $(gh_1)S(h_2)=g\epsilon(h)=(g(S(h_1))h_2), ~$$\forall h,g\in H.$
\end{definition}

\begin{definition}
A $\pi-$graded algebra ($\pi-$algebra) is a family $A=\{A_p\}_{p\in \pi}$ of $k-$spaces endowed with a family $m=\{m_{p,q}:A_p\otimes A_q\rightarrow A_{pq}\}_{p,q\in \pi}$ of $k-$ maps (the $Q-$graded multiplication) and a $k-$linear map $\mu:k\rightarrow A_e$ (the unit) such that (we denote $m_{p,q}(a\otimes b)\equiv ab, ~ \mu(1_k)\equiv 1)$ $$~a(bc)=(ab)c, ~a1=a=1a,~~\forall p,q,r\in \pi,~a\in A_p, ~b\in A_q, ~c\in A_r.$$
\end{definition}

\begin{definition}
Let $A$ be a $\pi-$graded algebra, a $\pi-$graded left $A-$module is a family $M=\{M_p\}_{p\in \pi}$ of $k-$spaces endowed with a family $\varphi=\{\varphi_{p,q}:A_p\otimes M_q\rightarrow M_{pq}\}_{p,q\in \pi}$ of $k-$maps such that(we denote $\varphi (a\otimes m)\equiv a\cdot m$)
$$(ab)\cdot m=a\cdot (b\cdot m),~1\cdot m=m,~\forall p,q,r\in \pi,~a\in A_p,~b\in A_q,~m\in M_r.$$

A map of $\pi-$graded left $A-$modules is a family $f=\{f_p:M_p\rightarrow M^{'}_p\}_{p\in \pi}$ of $k-$maps such that  $$f_{pq}(a\cdot m)=a\cdot f_q(m),~\forall p,q\in \pi,~a\in A_p, ~m\in M_q.$$
\end{definition}

\begin{definition}
A $\pi-$graded Hopf algebra (Hopf $\pi-$algebra) is a $\pi-$graded algebra $H=\{H_p, m_{p,q},\mu\}_{p,q\in \pi}$ endowed with a family $S=\{S_p:H_p\rightarrow H_{p^{-1}}\}_{p\in \pi}$ of $k-$linear maps (the antipode) such that each $H_p$ is a coassociative counitary coalgebra with comultiplication $\Delta_p$ and counit map $\epsilon_p$, and $\{m_{p,q}\}_{p,q\in \pi}$ and $\mu$ are coalgebra maps and (we denote $\Delta_p(h)\equiv h_{(1,p)}\otimes h_{(2,p)}$)
$$S_p(h_{(1,p)})h_{(2,p)}=\epsilon_p(h)1=h_{(1,p)}S_p(h_{(2,p)}),~\forall p\in \pi,~ h\in H_p.$$
\end{definition}

\begin{definition}
Given a $k-$space $M,$ a $k-$map $R:M\otimes M\rightarrow M\otimes M$ is said to be a solution of the Long-equation if
$$R^{12}R^{23}=R^{23}R^{12},$$

where $R^{12}=R\otimes I,$ $R^{23}=I\otimes R:M\otimes M\otimes M\rightarrow M\otimes M\otimes M.$
\end{definition}

\section{$Q$-graded Hopf quasigroups}
\def\theequation{2.\arabic{equation}}
\setcounter{equation} {0} \hskip\parindent

\begin{definition}
A $Q-$graded Hopf quasigroup $H$ is

(1) A $Q-$graded algebra, i.e a family $H=\{H_p\}_{p\in Q}$ of $k-$spaces endowed with a family $m=\{m_{p,q}:H_p\otimes H_q\rightarrow H_{pq}\}_{p,q\in Q}$ of $k-$ maps (the $Q-$graded multiplication) and a $k-$linear map $\mu: k\rightarrow H_e$ (the unit) such that $h1=h=1h,~ \forall p\in Q,~h\in H_p.$

(2) Each $H_p$ is a coassociative counitary coalgebra with comultiplication $\Delta_p$ and counit $\epsilon_p$,

(3) $\{m_{p,q}\}$ are coalgebra maps, i.e
\begin{equation}
(hg)_{(1,pq)}\otimes (hg)_{(2,pq)}=h_{(1,p)}g_{(1,q)}\otimes h_{(2,p)}g_{(2,p)}, ~\forall p,q\in Q,~h\in H_p,~g\in H_q, \label{2a}
\end{equation}
\begin{equation}
\epsilon_{pq}(hg)=\epsilon_p(h)\epsilon_q(g),~\forall p,q\in Q,~h\in H_p,~g\in H_q. \label{2b}
\end{equation}

(4) $\mu$ is coalgebra map, i.e
$$\Delta_e(1)=1\otimes 1,~\epsilon_e(1)=1_k.$$

(5) A family $S=\{S_p:H_p\rightarrow H_{p^{-1}}\}_{p\in Q}$ of $k-$linear maps (the antipode) such that
\begin{equation}
S_p(h_{(1,p)})(h_{(2,p)}g)=\epsilon_p(h)g=h_{(1,p)}(S_p(h_{(2,p)})g),~\forall p,q\in Q,~ h\in H_p,~ g\in H_q. \label{2c}
\end{equation}
\begin{equation}
(gS_p(h_{(1,p)}))h_{(2,p)}=g\epsilon_p(h)=(gh_{(1,p)})S_p(h_{(2,p)}),~ \forall p,q\in Q,~ h\in H_p, ~g\in H_q.  \label{2d}
\end{equation}

\end{definition}

$H$ is a $Q-$graded Hopf quasigroup, then

(1) $H$ is $flexible$ if $Q$ is $flexible$ and
\begin{equation}
h_{(1,p)}(gh_{(2,p)})=(h_{(1,p)}g)h_{(2,p)},~\forall p,q\in Q,~h\in H_p,~ g\in H_q.\label{2j}
\end{equation}

(2) $H$ is $alternative$ if $H$ is flexible, $Q$ is $alternative$ and
\begin{equation}
h_{(1,p)}(h_{(2,p)}g)=(h_{(1,p)}h_{(2,p)})g,~ h(g_{(1,p)}g_{(2,p)})=(hg_{(1,p)})g_{(2,p)}, ~\forall p,q\in Q,~h\in H_p,~ g\in H_q.\label{2k}
\end{equation}

(3)$H$ is $Moufang$ if $Q$ is $Moufang$ and
\begin{equation}
 h_{(1,p)}(g(h_{(2,p)}f))=((h_{(1,p)}g)h_{(2,p)})f, ~\forall p,q,r\in Q,~ h\in H_p,~g\in H_q,~  f\in H_r.\label{2l}
 \end{equation}

(4) $H$ is $commutative$ if $Q$ is $commutative$ and
 \begin{equation}
 hg=gh,~\forall p,q\in Q,~h\in H_p, ~ g\in H_q. \label{2m}
 \end{equation}

(5) $H$ is $cocommutative$ if
\begin{equation}
\Delta_p(h)=h_{(1,p)}\otimes h_{(2,p)}=h_{(2,p)}\otimes h_{(1,p)},~\forall p\in Q,~h\in H_p. \label{2n}
\end{equation}
Now we show an example of $Q-$graded Hopf quasigroups.

\begin{example}
For any quasigroup $Q$, we can construct a $Q-$graded Hopf quasigroup as follows: a family  $kQ=\{H_p=(kp)\}_{p\in Q}$of $k-$spaces, the mulitiplcation $m_{p,q}(k_1p\otimes k_2q)=k_1k_2pq,$ the unitary element $e\in Q$, comultiplication $\Delta_p(p)=p\otimes p,$ counit $\epsilon_p(p)=1_k,$ antipode $S_p(p)=p^{-1}.$
\end{example}
\begin{proposition}
$H$ is a $Q-$graded Hopf quasigroup, then

(1) \begin{equation}
S_p(h_{(1,p)})h_{(2,p)}=\epsilon_p(h)1=h_{(1,p)}S_p(h_{(2,p)}),~ \forall p\in Q,~h\in H_p. \label{2e}
\end{equation}

(2) \begin{equation}
 S_{pq}(hg)=S_q(g)S_p(h), ~\forall p,q\in Q,~h\in H_p,g\in H_q. \label{2f}
 \end{equation}

(3) \begin{equation}S_p(h)_{(1,p^{-1})}\otimes S_p(h)_{(2,p^{-1})}=\Delta_{p^{-1}}(S_p(h))=S_{p}(h_{(2,p)})\otimes S_p(h_{(1,p)}),~  \forall p\in Q,~h\in H_p.\label{2g}
\end{equation}

(4) \begin{equation}
S_e(1)=1.\label{2h}
\end{equation}

(5) \begin{equation}
\epsilon_{p^{-1}} S_p(h)=\epsilon_p(h),~\forall p\in Q,~h\in H_p. \label{2i}
\end{equation}

\end{proposition}

\begin{proof} (1) and (4) are obtained by the definition of $Q-$graded Hopf quasigroup.

(2)$$\aligned
&S_q(g)S_p(h)=S_q(g_{(1,q)}\epsilon_q(g_{(2,q)}))S_p(h_{(1,p)}\epsilon_p(h_{(2,p)})) \\
&\stackrel{(\ref{2b})}{=} S_q(g_{(1,q)})S_p(h_{(1,p)})\epsilon_{pq}(h_{(2,p)}g_{(2,q)})\\
&=S_q(g_{(1,q)})((S_p(h_{(1,p)})(h_{(2,p)}g_{(2,q)})_{(1,pq)})S_{pq}((h_{(2,p)}g_{(2,p)})_{(2,pq)}))\\
&=S_q(g_{(1,q)})((S_p(h_{(1,p)})(h_{(2,p)}g_{(2,q)}))S_{pq}(h_{(3,p)}g_{(3,q)}))\\
&=S_q(g_{(1,q)(1,q)})((S_p(h_{(1,p)(1,p)})(h_{(1,p)(2,p)}g_{(1,q)(2,q)}))S_{pq}(h_{(2,p)}g_{(2,q)}))\\
&\stackrel{(\ref{2c})}{=}S_q(g_{(1,q)(1,q)})(\epsilon_p(h_{(1,p)})g_{(1,q)(2,q)}S_{pq}(h_{(2,p)}g_{(2,q)}))\\
&=S_q(g_{(1,q)(1,q)})(g_{(1,q)(2,q)}S_{pq}(\epsilon_{p}(h_{(1,p)})h_{(2,p)}g_{(2,q)}))\\
&=S_q(g_{(1,q)(1,q)})(g_{(1,q)(2,q)}S_{pq}(hg_{(2,q)}))\\
&\stackrel{(\ref{2c})}{=}\epsilon_q(g_{(1,q)})S_{pq}(hg_{(2,q)})\\
&=S_{pq}(h\epsilon_q(g_{(1,q)})g_{(2,q)})=S_{pq}(hg).
\endaligned$$

(3)$$\aligned
&\Delta_{p^{-1}}(S_p(h))=S_p(h_{(2,p)})_{(1,p^{-1})}\otimes \epsilon_p(h_{(1,p)})S_p(h_{(2,p)})_{(2,p^{-1})}\\
&\stackrel{(\ref{2c})}=S_p(h_{(2,p)})_{(1,p^{-1})}\otimes S_p(h_{(1,p)(1,p)})(h_{(1,p)(2,p)}S_p(h_{(2,p)})_{(2,p^{-1})})\\
&=\epsilon_p(h_{(1,p)(1,p)(2,p)})S(h_{(2,p)})_{(1,p^{-1})}\\
&\otimes S_p(h_{(1,p)(1,p)(1,p)})(h_{(1,p)(2,p)}S_p(h_{(2,p)})_{(2,p^{-1})})\\ &\stackrel{(\ref{2c})}=S_p(h_{(1,p)(1,p)(2,p)(1,p)})(h_{(1,p)(1,p)(2,p)(2,p)}S_p(h_{(2,p)})_{(1,p^{-1})})\\
&\otimes S_p(h_{(1,p)(1,p)(1,p)})(h_{(1,p)(2,p)}S_p(h_{(2,p)})_{(2,p^{-1})})\\
&=S_p(h_{(2,p)})(h_{(3,p)}S_p(h_{(5,p)})_{(1,p^{-1})})\otimes S_p(h_{(1,p)})(h_{(4,p)}S_p(h_{(5,p)})_{(2,p^{-1})}))\\
&=S_p(h_{(1,p)(2,p)})(h_{(2,p)(1,p)(1,p)}S_p(h_{(2,p)(2,p)})_{(1,p^{-1})})\\
&\otimes S_p(h_{(1,p)(1,p)})(h_{(2,p)(1,p)(2,p)}S_p(h_{(2,p)(2,p)})_{(2,p^{-1})})\\
&\stackrel{(\ref{2a})}=S_p(h_{(1,p)(2,p)})(h_{(2,p)(1,p)}S_p(h_{(2,p)(2,p)}))_{(1,e)}
\otimes S_p(h_{(1,p)(1,p)})(h_{(2,p)(1,p)}S_p(h_{(2,p)(2,p)}))_{(2,e)}\\
&\stackrel{(\ref{2e})}=S_p(h_{(1,p)(2,p)})\epsilon_p(h_{(2,p)})\otimes S_p(h_{(1,p)(1,p)})\\
&=S_p(h_{(2,p)})\otimes S_p(h_{(1,p)}).
\endaligned$$

(5)
$$\begin{aligned}
&(\epsilon_{p^{-1}} S_p)(h)\\
&=\epsilon_{p^{-1}} (S_p(h_{(1,p)}))\epsilon_p(h_{(2,p)})\\
&\stackrel{(\ref{2b})}=\epsilon_{e}(S_p(h_{(1,p)})h_{(2,p)})\\
&=\epsilon_e(1\epsilon_p(h))\\
&=\epsilon_p(h).
\end{aligned}$$

\end{proof}
\begin{proposition}
Let $H$ be a $Q-$graded Hopf quasigroup, if $H$ is $commutative$ or $cocommutative$, then $S_{p^{-1}}S_p(h)=h,~\forall p\in Q,~h\in H_p.$
\end{proposition}

\begin{proof}
If $H$ is commutative,
$$\aligned
&S_{p^{-1}}S_p(h)=S_{p^{-1}}(S_p(h_{(1,p)}))\epsilon_p(h_{(2,p)})\\
&\stackrel{(\ref{2e})}=S_{p^{-1}}(S_p(h_{(1,p)}))(S_p(h_{(2,p)})h_{(3,p)})\\
&=S_{p^{-1}}(S_p(h_{(1,p)(1,p)}))(S_p(h_{(1,p)(2,p)})h_{(2,p)})\\
&\stackrel{(\ref{2g})}=S_{p^{-1}}(S_p(h_{(1,p)})_{(2,p^{-1})})(S_p(h_{(1,p)})_{(1,p^{-1})}h_{(2,p)})\\
&\stackrel{(\ref{2m})}=(h_{(2,p)}S_p(h_{(1,p)})_{(1,p^{-1})})S_{p^{-1}}(S_p(h_{(1,p)})_{(2,p^{-1})})\\
&\stackrel{(\ref{2d})}=h_{(2,p)}\epsilon_{p^{-1}}(S_p(h_{(1,p)}))\\
&\stackrel{(\ref{2i})}=h_{(2,p)}\epsilon_p(h_{(1,p)})=h,\\
\endaligned$$

If $H$ is cocommutative, $$\aligned
&S_{p^{-1}}S_p(h)=S_{p^{-1}}(S_p(h_{(1,p)}))\epsilon_p(h_{(2,p)})\\
&\stackrel{(\ref{2e})}=S_{p^{-1}}(S_p(h_{(1,p)}))(S_p(h_{(2,p)})h_{(3,p)})\\
&=S_{p^{-1}}(S_p(h_{(1,p)(1,p)}))(S_p(h_{(1,p)(2,p)})h_{(2,p)})\\
&\stackrel{(\ref{2g})}=S_{p^{-1}}(S_p(h_{(1,p)})_{(2,p^{-1})})(S_p(h_{(1,p)})_{(1,p^{-1})}h_{(2,p)})\\
&\stackrel{(\ref{2n})}=S_{p^{-1}}(S_p(h_{(1,p)})_{(1,p^{-1})})(S_p(h_{(1,p)})_{(2,p^{-1})}h_{(2,p)})\\
&\stackrel{(\ref{2c})}=\epsilon_{p^{-1}}(S_p(h_{(1,p)}))h_{(2,p)}\\
&\stackrel{(\ref{2i})}=\epsilon_{p}(h_{(1.p)})h_{(2,p)}=h.\\
\endaligned$$
\end{proof}

\begin{proposition}
Let $H$ be a $Moufang$ $Q-$graded Hopf quasigroup such that $S_p$ is invertible for any $p\in Q$, then $\forall p,q,r\in Q,~h\in H_p,~g\in H_q,~f\in H_r,$ the following three conditions are equivalent :

(1) $h_{(1,p)}(g(h_{(2,p)}f))=((h_{(1,p)}g)h_{(2,p)})f$,

(2) $((hg_{(1,q)})f)g_{(2,q)}=h(g_{(1,p)}(fg_{(2,q)}))$,

(3) $(h_{(1,p)}g)(fh_{(2,p)})=(h_{(1,p)}(gf))h_{(2,p)}$.
\end{proposition}

\begin{proof}
$(1)\Rightarrow(2)$
$$\aligned
&((S_r(f)S_p(h)_{(1,p^{-1})})S_q(g))S_p(h)_{(2,p^{-1})}\\
&\stackrel{(\ref{2g})}=((S_r(f)S_p(h_{(2,p)}))S_q(g))S_p(h_{(1,p)})\\
&\stackrel{(\ref{2f})}=S_{p(q(pr))}(h_{(1,p)}(g(h_{(2,p)}f)))\\
&=S_{((pq)p)r}((h_{(1,p)}g)h_{(2,p)})f\\
&\stackrel{(\ref{2f})}=S_r(f)(S_q(h_{(2,p)})(S_q(g)S_p(h_{(1,p)})))\\
&\stackrel{(\ref{2g})}=S_r(f)(S_q(h)_{(1,p^{-1})}(S_q(g)S_p(h)_{(2,p^{-1})}))\\
\endaligned$$

using $S_p$ are invertible and $H$ is $Moufang$, the above equation is equivalent to
$$((hg_{(1,q)})f)g_{(2,q)}=h(g_{(1,p)}(fg_{(2,q)})),~\forall p,q,r\in Q,~h\in H_p,~g\in H_q,~f\in H_r.$$

$(2)\Rightarrow(1)$ The proof is similarly to $(1)\Rightarrow(2)$.

$(1)\Rightarrow (3)$
$$\aligned
&h(gf)=h_{(1,p)}\epsilon_p(h_{(2,p)})(gf)\\
&=h_{(1,p)}(g\epsilon_p(h_{(2,p)})f)\\
&\stackrel{(\ref{2c})}=h_{(1,p)}(g(h_{(2,p)(1,p)}(S_p(h_{(2,p)(2,p)})f)))\\
&=h_{(1,p)(1,p)}(g(h_{(1,p)(2,p)}(S_p(h_{(2,p)})f)))\\
&=((h_{(1,p)(1,p)}g)h_{(1,p)(2,p)})(S_p(h_{(2,p)})f),\\
\endaligned$$

replacing $g$ with $gf_{(1,r)}$ and $f$ with $S_r(f_{(2,r)})$, we have:
$$hg\epsilon_r(f)\stackrel{(\ref{2d})}=h((gf_{(1,r)})S_r(f_{(2,r)}))\stackrel{(\ref{2f})}=((h_{(1,p)(1,p)}(gf_{(1,r)}))h_{(1,p)(2,p)})S_{rp}(f_{(2,r)}h_{(2,p)}),$$

replacing $h$ with $h_{(1,p)}$ and $f$ with $f_{(1,r)}$, multiply on the right by $f_{(2,r)}h_{(2,p)}$, we have:
$$\epsilon_r(f_{(1,r)})(h_{(1,p)}g)(f_{(2,r)}h_{(2,p)})$$
$$=(((h_{(1,p)(1,p)(1,p)}(gf_{(1,r)(1,r)}))h_{(1,p)(1,p)(2,p)}
)S_{rp}(f_{(1,r)(2,r)}h_{(1,p)(2,p)}))(f_{(2,r)}h_{(2,p)}),$$

the left side of the above equation equals $(h_{(1,p)}g)(fh_{(2,p)}),$ now consider the right side of the equation: by the coassociativity and $\Delta_p$ ia algebra map, the right side equals:
$$\aligned
&=(((h_{(1,p)}(gf_{(1,r)}))h_{(2,p)})S_{rp}(f_{(2,r)}h_{(3,p)}))(f_{(3,r)}h_{(4,p)})\\
&\stackrel{(\ref{2a})}=(((h_{(1,p)}(gf_{(1,r)}))h_{(2,p)})S_{rp}((f_{(2,r)}h_{(3,p)})_{(1,rp)}))((f_{(2,r)}h_{(3,p)})_{(2,rp)})\\
&\stackrel{(\ref{2d})}=((h_{(1,p)}(gf_{(1,r)}))h_{(2,p)})\epsilon_{rp}(f_{(2,r)}h_{(3,p)})\\
&\stackrel{(\ref{2b})}=((h_{(1,p)}(gf_{(1,r)}))h_{(2,p)})\epsilon_p(h_{(3,p)})\epsilon_r(f_{(2,r)})\\
&=(h_{(1,p)}(gf))h_{(2,p)},\\
\endaligned$$

so,
$$(h_{(1,p)}g)(fh_{(2,p)})=(h_{(1,p)}(gf))h_{(2,p)},~\forall p,q,r\in Q,~h\in H_p,~g\in H_q,~f\in H_r.$$

$(3)\Rightarrow (1)$ Assume (3) holds, then
$$((h_{(1,p)(1,p)}g)(f_{(1,r)}h_{(1,p)(2,p)}))S_{rp}(f_{(2,r)}h_{(2,p)})$$
$$=((h_{(1,p)(1,p)}(gf_{(1,r)}))h_{(1,p)(2,p)})S_{rp}(f_{(2,r)}h_{(2,p)}),$$

the left side of above equation equals:
$$\aligned
&=((h_{(1,p)}g)(f_{(1,r)}h_{(2,p)}))S_{rp}(f_{(2,r)}h_{(3,p)})\\
&\stackrel{(\ref{2a})}=((h_{(1,p)}g)(fh_{(2,p)})_{(1,rp)})S_{rp}((fh_{(2,p)})_{(2,rp)})\\
&\stackrel{(\ref{2d})}=(h_{(1,p)}g)\epsilon_{rp}(fh_{(2,p)})\\
&\stackrel{(\ref{2b})}=(h_{(1,p)}g)\epsilon_r(f)\epsilon_p(h_{(2,p)})\\
&=(hg)\epsilon_r(f),\\
\endaligned$$

so, we have:
$$hg\epsilon_r(f)=((h_{(1,p)(1,p)}(gf_{(1,r)}))h_{(1,p)(2,p)})(S_p(h_{(2,p)})S_r(f_{(2,r)})),$$

replacing $g$ with $gS_r(f_{(1,r)})$ and $f$ with $f_{(2,r)}$,
$$h(gS_r(f_{(1,r)}))\epsilon_r(f_{(2,r)})=((h_{(1,p)(1,p)}((gS_r(f_{(1,r)}))f_{(2,r)(1,r)}))h_{(1,p)(2,p)})
(S_p(h_{(2,p)})S_r(f_{(2,r)(2,r)})),$$

so,
$$h(gS_r(f))=((h_{(1,p)(1,p)}g)h_{(1,p)(2,p)})(S_p(h_{(2,p)})S_r(f)),$$

replacing $h$ with $h_{(1,p)}$ and $f$ with $S^{-1}_{pr}(h_{(2,p)}f)$, we have:
$$h_{(1,p)}(g(h_{(2,p)}f))=((h_{(1,p)(1,p)(1,p)}g)h_{(1,p)(1,p)(2,p)})(S_{p}(h_{(1,p)(2,p)})(h_{(2,p)}f))
\stackrel{(\ref{2c})}=((h_{(1,p)}g)h_{(2,p)})f.$$

\end{proof}

\begin{lemma}
Let $H$ be a cocommutative and flexible $Q-$graded Hopf quasigroup, then
$$h_{(1,p)}(gS_p(h_{(2,p)}))=(h_{(1,p)}g)S_p(h_{(2,p)}), ~\forall p,q\in Q,~h\in H_p,~g\in H_q.$$
\end{lemma}

\begin{proof}
$$\begin{aligned}
&(h_{(1,p)}(gS_p(h_{(1,p)(2,p)})))h_{(2,p)}\\
&\stackrel{(\ref{2n})}=(h_{(1,p)}(gS_p(h_{(3,p)})))h_{(2,p)}\\
&\stackrel{(\ref{2j})}=h_{(1,p)}((gS_p(h_{(3,p)}))h_{(2,p)})\\
&\stackrel{(\ref{2n})}=h_{(1,p)}((gS_p(h_{(2,p)(1,p)}))h_{(2,p)(2,p)})\\
&\stackrel{(\ref{2d})}=h_{(1,p)}(g\epsilon_p(h_{(2,p)}))\\
&=(h_{(1,p)}g)\epsilon_p(h_{(2,p)})\\
&\stackrel{(\ref{2d})}=((h_{(1,p)}g)S_p(h_{(2,p)(1,p)}))h_{(2,p)(2,p)}\\
&=(h_{(1,p)(1,p)}g)S_p(h_{(1,p)(2,p)})h_{(2,p)},\\
\end{aligned}$$

So,
$$h_{(1,p)}(gS_p(h_{(2,p)}))=(h_{(1,p)}g)S_p(h_{(2,p)}),~\forall p,q\in Q,~h\in H_p,~g\in H_q.$$
\end{proof}

Therfore we have a notion of left adjoint action of a $Q-$graded Hopf quasigroup $H$ when it is cocommutative and flexible.

\begin{definition}
Let $Q$ be an associative quasigroup and $H$ be a $Q-$graded Hopf quasigroup, an associator $\delta$ be defined by
 $$(hg)f=\delta(h_{(1,p)},g_{(1,p)},f_{(1,r)})(h_{(2,p)}(g_{(2,p)}f_{(2,r)})),~\forall p,q,r\in Q,~h\in H_p,~g\in H_q,~f\in H_r.$$
\end{definition}

\begin{theorem}Let $Q$ be an associative quasigroup and $H$ be a $Q-$graded Hopf quasigroup,

(1) the associator $\delta$ exists and is uniquely determined as
$$\delta(h,g,f)=((h_{(1,p)}g_{(1,q)})f_{(1,r)})S_{pqr}(h_{(2,p)}(g_{(2,q)}f_{(2,r)})),~\forall p,q,r\in Q,~h\in H_p,~g\in H_q,~f\in H_r.$$

(2) $\delta(1,h,g)=\delta(h,1,g)=\delta(h,g,1)=\epsilon_p(h)\epsilon_q(g)1.$

(3)$\delta(h_{(1,p)},S_p(h_{(2,p)}),g)=\delta(S_p(h_{(1,p)}),h_{(2,p)},g)=\epsilon_p(h)\epsilon_q(g)1,$

$\delta(h,g_{(1,q)},S_q(g_{(2,q)}))=\delta(h, S_q(g_{(1,q)}),g_{(2,q)})=\epsilon_p(h)\epsilon_q(g)1,$

$\delta(h_{(1,p)}g_{(1,q)},S_q(g_{(2,q)}),S_p(h_{(2,p)}))=\delta(S_p(h_{(1,p)})S_q(g_{(1,q)}),g_{(2,q)},h_{(2,p)})
=\epsilon_p(h)\epsilon_q(g)1,$

$\delta(S_p(h_{(1,p)}), S_q(g_{(1,q)}), g_{(2,q)}h_{(2,p)})=\delta(h_{(1,p)},g_{(1,p)},S_q(g_{(2,q)})S_p(h_{(2,p)}))=\epsilon_p(h)\epsilon_q(g)1,$

$\delta(S_p(h_{(1,p)}),h_{(2,p)}S_q(g_{(1,q)}),g_{(2,q)})=\delta(h_{(1,p)},S_p(h_{(2,p)})g_{(1,q)},S_q(g_{(2,q)}))=\epsilon_p(h)\epsilon_q(g)1.$
\end{theorem}
\begin{proof} (1)
We suppose $\delta$ exists then applying it to the rebracket, then
$$\aligned
&((h_{(1,p)}g_{(1,q)})f_{(1,r)})S_{pqr}(h_{(2,p)}(g_{(2,q)}f_{(2,r)}))\\
&=(\delta(h_{(1,p)(1,p)},g_{(1,q)(1,q)},f_{(1,r)(1,r)})(h_{(1,p)(2,p)}(g_{(1,q)(2,q)}f_{(1,r)(2,r)})))
S_{pqr}(h_{(2,p)}(g_{(2,q)}f_{(2,r)}))\\
&\stackrel{(\ref{2a})}=(\delta(h_{(1,p)},g_{(1,q)},g_{(1,r)})(h_{(2,p)}(g_{(2,q)}f_{(2,r)}))_{(1,pqr)})
S_{pqr}((h_{(2,p)}(g_{(2,q)}f_{(2,r)}))_{(2,pqr)})\\
&\stackrel{(\ref{2d})}=\delta(h,g,f),\\
\endaligned$$

we verify similarily,
$$\aligned
&\delta(h_{(1,p)},g_{(1,q)},f_{(1,r)})(h_{(2,p)}(g_{(2,q)}f_{(2,r)}))\\
&=(((h_{(1,p)(1,p)}g_{(1,q)(1,q)})f_{(1,r)(1,r)})S_{pqr}(h_{(1,p)(2,p)}(g_{(1,q)(2,q)}f_{(1,r)(2,r)})))(h_{(2,p)}(g_{(2,q)}f_{(2,r)}))\\
&=(((h_{(1,p)}g_{(1,p)})f_{(1,r)})S_{pqr}(h_{(2,p)(1,p)}(g_{(2,q)(1,q)}f_{(2,r)(1,r)})))(h_{(2,p)(2,p)}(g_{(2,q)(2,q)}f_{(2,r)(2,r)}))\\
&=(((h_{(1,p)}g_{(1,p)})f_{(1,r)})S_{pqr}((h_{(2,p)}(g_{(2,q)}f_{(2,r)}))_{(1,pqr)}))(h_{(2,p)}(g_{(2,q)}f_{(2,r)}))_{(2,pqr)}\\
&\stackrel{(\ref{2d})}=((h_{(1,p)}g_{(1,p)})f_{(1,r)})\epsilon_{pqr}(h_{(2,p)}(g_{(2,q)}f_{(2,r)}))\stackrel{(\ref{2b})}=(hg)f.
\endaligned$$

(2) $$\delta(1,h,g)=(h_{(1,p)}g_{(1,q)})S_{pq}(h_{(2,p)}g_{(2,q)})=(hg)_{(1,pq)}S_{pq}((hg)_{(2,pq)})\stackrel{(\ref{2e})}=\epsilon_p(h)\epsilon_q(g)1,$$

the rest are similar.

(3)
$$\aligned
&\delta(h_{(1,p)},h_{(2,p)},g)=((h_{(1,p)(1,p)}S_p(h_{(2,p)})_{(1,p^{-1})})g_{(1,q)})S_q(h_{(1,p)(2,p)}(S_{p}(h_{(2,p)})_{(2,p^{-1})}g_{(2,q)}))\\
&\stackrel{(\ref{2g})}=((h_{(1,p)(1,p)}S_p(h_{(2,p)(2,p)}))g_{(1,q)})S_q(h_{(1,p)(2,p)}(S_p(h_{(2,p)(1,p)})g_{(2,q)}))\\
&=((h_{(1,p)(1,p)}S_p(h_{(2,p)}))g_{(1,q)})S_q(h_{(1,p)(2,p)(1,p)}(S_p(h_{(1,p)(2,p)(2,p)})g_{(2,q)}))\\
&\stackrel{(\ref{2c})}=((h_{(1,p)}S_p(h_{(2,p)}))g_{(1,q)})S_q(g_{(2,q)})\\
&\stackrel{(\ref{2e})}=\epsilon_p(h)g_{(1,q)}S_q(g_{(2,q)})\stackrel{(\ref{2e})}=\epsilon_p(h)\epsilon_q(g)1,\\
\endaligned$$

the rest are similar.
\end{proof}

\section{The first fundamental theorem for $Q-$graded Hopf quasigroups}
\def\theequation{3.\arabic{equation}}
\setcounter{equation} {0} \hskip\parindent

\begin{definition}
Let $H=\{H_p\}_{p\in Q}$ be a (not necessary associative) $Q-$graded algebra and for any $p\in Q$, $H_p$ is a (not necessary coassociative) coalgebra $(H_p,\Delta_p,\epsilon_p)$ such that for any $p,q\in Q$, $m_{p,q},~\mu$ are maps of coalgebras, then

(1) A family $\beta=\{\beta_{p,q}:H_p\otimes H_q\rightarrow H_{pq}\otimes H_q,\beta_{p,q}(h\otimes g)=hg_{(1,q)}\otimes g_{(2,q)}\}_{p,q\in Q}$ of $k-$maps is called the right $Galois$ map of $H$.

(2) A family $\gamma=\{\gamma_{p,q}:H_p\otimes H_q\rightarrow H_p\otimes H_{pq},\gamma_{p,q}(h\otimes g)=h_{(1,p)}\otimes h_{(2,p)}g\}_{p,q\in Q}$ of $k-$maps is called the left $Galois$ map of $H$.

(3) A family $\phi=\{\phi_{p,q}:H_p\otimes H_q\rightarrow H_{pq}\otimes H_q\}_{p,q\in Q}$ of $k-$maps is almost left $H-$linear if
$$
\phi_{p,q}(h\otimes g)=(h\otimes 1)\phi_{e,q}(1\otimes g),~\forall p,q\in Q,~h\in H_p,~g\in H_q.
$$

(4) A family $\phi=\{\phi_{p,q}:H_p\otimes H_q\rightarrow H_{pq}\otimes H_q\}_{p,q\in Q}$ of $k-$maps is almost right $H-$colinear if
$$
\phi_{p,q}(h\otimes g)=(I\otimes \epsilon_q)\phi_{p,q}(h\otimes g_{(1,q)})\otimes g_{(2,q)},~\forall p,q\in Q,~h\in H_{p},~g\in H_q.
$$

(5) A family $\phi=\{\phi_{p,q}:H_p\otimes H_q\rightarrow H_{pq}\otimes H_q\}_{p,q\in Q}$ of $k-$maps is right $H-$colinear if
$$
\phi_{p,q}(h\otimes g_{(1,q)})\otimes g_{(2,q)}=(I\otimes \Delta_{q})\phi_{p,q}(h\otimes g),~\forall p,q\in Q,~h\in H_p,~g\in H_q.
$$

(6) A family $\phi=\{\phi_{p,q}:H_p\otimes H_q\rightarrow H_{p}\otimes H_{pq}\}_{p,q\in Q}$ of $k-$maps is almost right $H-$linear if
$$
\phi_{p,q}(h\otimes g)=\phi_{p,e}(h\otimes 1)(1\otimes g),~\forall p,q\in Q,~h\in H_p,~g\in H_q.
$$

(7) A family $\phi=\{\phi_{p,q}:H_p\otimes H_q\rightarrow H_{p}\otimes H_{pq}\}_{p,q\in Q}$ of $k-$maps is almost left $H-$colinear if
$$
 \phi_{p,q}(h\otimes g)=h_{(1,p)}\otimes (\epsilon_p\otimes I)\phi_{p,q}(h_{(2,p)}\otimes g),~\forall p,q\in Q,~h\in H_p,~g\in H_q.
$$

(8) A family $\phi=\{\phi_{p,q}:H_p\otimes H_q\rightarrow H_{p}\otimes H_{pq}\} _{p,q\in Q}$ of $k-$maps is left $H-$colinear if
$$
 h_{(1,p)}\otimes \phi_{p,q}(h_{(2,p)}\otimes g)=(\Delta_p\otimes I)\phi_{p,q}(h\otimes g),~\forall p,q\in Q,~h\in H_p,~g\in H_q.
$$

(9) A family $\phi=\{\phi_{p,q}:H_{p}\otimes H_q\rightarrow H_{pq^{-1}}\otimes H_q\}_{p,q\in Q}$ of $k-$maps is almost left $H-$linear if
$$
 \phi_{p,q}(h\otimes g)=(h\otimes 1)\phi_{e,q}(1\otimes g),~\forall p,q\in Q,~h\in H_{p},~g\in H_q.
$$

(10) A family $\phi=\{\phi_{p,q}:H_{p}\otimes H_q\rightarrow H_{pq^{-1}}\otimes H_q\}_{p,q\in Q}$ of $k-$maps is almost right $H-$colinear if
$$
\phi_{p,q}(h\otimes g)=(I\otimes \epsilon_q)\phi_{p,q}(h\otimes g_{(1,q)})\otimes g_{(2,q)},~\forall p,q\in Q,~h\in H_{p},~g\in H_q.
$$

(11) A family $\phi=\{\phi_{p,q}:H_{p}\otimes H_q\rightarrow H_{pq^{-1}}\otimes H_q\}_{p,q\in Q}$ of $k-$maps is right $H-$colinear if
$$
 \phi_{p,q}(h\otimes g_{(1,q)})\otimes g_{(2,q)}=(I\otimes \Delta_{q})\phi_{p,q}(h\otimes g),~\forall p,q\in Q,~h\in H_{p},~g\in H_q.
$$

(12) A family $\phi=\{\phi_{p,q}:H_p\otimes H_{q}\rightarrow H_{p}\otimes H_{p^{-1}q}\}_{p,q\in Q}$ of $k-$maps is almost right $H-$linear if
$$
\phi_{p,q}(h\otimes g)=\phi_{p,e}(h\otimes 1)(1\otimes g),\forall p,q\in Q,~h\in H_p,~g\in H_{q}.
$$

(13) A family $\phi=\{\phi_{p,q}:H_p\otimes H_{q}\rightarrow H_{p}\otimes H_{p^{-1}q}\}_{p,q\in Q}$ of $k-$maps is almost left $H-$colinear if
$$
\phi_{p,q}(h\otimes g)=h_{(1,p)}\otimes (\epsilon_p\otimes I)\phi_{p,q}(h_{(2,p)}\otimes g),~\forall p,q\in Q,~h\in H_p,~g\in H_{q}.
$$

(14) A family $(\phi=\{\phi_{p,q}:H_p\otimes H_{p}\rightarrow H_{p}\otimes H_{p^{-1}q}\} _{p,q\in Q}$of $k-$maps is left $H-$colinear if
$$
h_{(1,p)}\otimes \phi_{p,q}(h_{(2,p)}\otimes g)=(\Delta_p\otimes I)\phi_{p,q}(h\otimes g),~\forall p,q\in Q,~h\in H_p,~g\in H_{q}.
$$

\end{definition}

\begin{lemma}
If for any $p\in Q,$ $\Delta_p$ is a coassociative, then the notion of almost $H-$colinearity coincides with that of $H-$colinearity.
\end{lemma}

\begin{proof}
Suppose a family $\phi=\{\phi_{p,q}:H_p\otimes H_q\rightarrow H_{p}\otimes H_{pq}\}_{p,q\in Q}$ of $k-$maps is left $H-$colinear,
$$\aligned
&\phi_{p,q}(h\otimes g)\\
&=(I\otimes \epsilon_p\otimes I)(\Delta_p\otimes I)\phi_{p,q}(h\otimes g)\\
&=(I\otimes \epsilon_p\otimes I)(h_{(1,p)}\otimes \phi_{p,q}(h_{(2,p)}\otimes g))\\
&=h_{(1,p)}\otimes (\epsilon_p\otimes I)\phi_{p,q}(h_{(2,p)}\otimes g),~\forall p,q\in Q,~h\in H_p,~g\in H_q.\\
\endaligned$$

therefore $\phi$ is almost left $H-$colinear. Similarly, suppose a family $\phi=\{\phi_{p,q}:H_p\otimes H_q\rightarrow H_{p}\otimes H_{pq}\}_{p,q\in Q}$ of $k-$maps is almost left $H-$colinear,
$$\begin{aligned}
&(\Delta_p\otimes I)\phi_{p,q}(h\otimes g)\\
&=(\Delta_p\otimes I)(h_{(1,p)}\otimes (\epsilon_p\otimes I)\phi(h_{(2,p)}\otimes g))\\
&=h_{(1,p)(1,p)}\otimes h_{(1,p)(2,p)}\otimes (\epsilon_p\otimes I)\phi(h_{(2,p)}\otimes g)\\
&=h_{(1,p)}\otimes h_{(2,p)(1,p)}\otimes (\epsilon_p\otimes I)\phi(h_{(2,p)(2,p)}\otimes g)\\
&=h_{(1,p)}\otimes \phi_{p,q}(h_{(2,p)}\otimes g),~\forall p,q\in Q,~h\in H_p,~g\in H_q,\\
\end{aligned}$$

therefore $\phi$ is left $H-$colinear.
Analogously we can get the other results.
\end{proof}

\begin{lemma}
Let $H$ be a $Q-$graded Hopf quasigroup, then the right Galois map $\beta$ is almost left $H-$linear and almost right $H-$colinear, the left Galois map $\gamma$ is almost right $H-$linear and almost left $H-$colinear.
\end{lemma}

\begin{proof}
Since
$$\beta_{p,q}(h\otimes g)=hg_{(1,p)}\otimes g_{(2,p)}=(h\otimes 1)(g_{(1,q)}\otimes g_{(2,q)})$$$$=(h\otimes 1)\beta_{e,q}(1\otimes g),~\forall p,q\in Q,~h\in H_p,~g\in H_q.$$

and
$$\beta_{p,q}(h\otimes g)=hg_{(1,q)}\otimes g_{(2,q)}=hg_{(1,q)}\epsilon_q(g_{(2,q)})\otimes g_{(3,q)}=(I\otimes \epsilon_q)(hg_{(1,q)}\otimes g_{(2,q)})\otimes g_{(3,q)}$$$$=(I\otimes \epsilon_q)\beta_{p,q}(h\otimes g_{(1,q)})\otimes g_{(2,q)},~\forall p,q\in Q,~h\in H_p,~g\in H_q,$$

$\beta$ is almost left $H-$linear and almost $H-$colinear. Similarly we can get the result for $\gamma$.

\end{proof}
\begin{theorem}
Let $H=\{H_p\}_{p\in Q}$ be a (not necessary associative) $Q-$graded algebra and for any $p\in Q$, $H_p$ is a (not necessary coassociative) coalgebra $(H_p,\Delta_p,\epsilon_p)$ such that for any $p,q\in Q$, $m_{p,q},~\mu$ are maps of coalgebras. Then $H$ is a $Q-$graded Hopf quasigroup $\Longleftrightarrow$

 (1) For any $p\in Q$, $\Delta_p$ is coassociative .

 (2) There is a family $\beta^{\star}=\{\beta^{\star}_{p,q}:H_p\otimes H_q\rightarrow H_{pq^{-1}}\otimes H_q\}_{p,q\in Q}$ of $k-$maps which is almost left $H-$linear and satisfy
$$\beta^{\star}_{pq,q}\beta_{p,q}(h\otimes g)=h\otimes g,~\forall p,q\in Q,~h\in H_p,~g\in H_q,$$
$$\beta_{pq^{-1},q}\beta^{\star}_{p,q}(h\otimes g)=h\otimes g,~\forall p,q\in Q,~h\in H_p,~g\in H_q.$$

(3) There is a family $\gamma^{\star}=\{\gamma^{\star}_{p,q}:H_p\otimes H_q\rightarrow H_p\otimes H_{p^{-1}q}\}_{p,q\in Q}$ of $k-$maps which is almost right $H-$linear and satisfy
$$\gamma^{\star}_{p,pq}\gamma_{p,q}(h\otimes g)=h\otimes g,~\forall p,q\in Q,~h\in H_p,~g\in H_q,$$
$$\gamma_{p,p^{-1}q}\gamma^{\star}_{p,q}(h\otimes g)=h\otimes g,~\forall p,q\in Q,~h\in H_p,~g\in H_q.$$
\end{theorem}

\begin{proof}
$(\Longrightarrow)$ $H$ is a $Q-$graded Hopf quasigroup, by definition, $\forall p\in Q,$ $\Delta_p$ is coassociative. We set
 $$\beta^{\star}=\{\beta^{\star}_{p,q}:H_{p}\otimes H_q\rightarrow H_{pq^{-1}}\otimes H_q,\beta^{\star}_{p,q}(h\otimes g)=hS_q(g_{(1,q)})\otimes g_{(2,q)}\}_{p,q\in Q}$$
 and
 $$\gamma^{\star}=\{\gamma^{\star}_{p,q}:H_p\otimes H_{q}\rightarrow H_p\otimes H_{p^{-1}q}, \gamma^{\star}_{p,q}(h\otimes g)=h_{(1,p)}\otimes S_p(h_{(2,p)})g\}_{p,q\in Q}.
$$

Since
$$\beta^{\star}_{p,q}(h\otimes g)=hS_q(g_{(1,q)})\otimes g_{(2,q)}=(h\otimes 1)(S_q(g_{(1,q)})\otimes g_{(2,q)})$$$$=(h\otimes 1)\beta^{\star}_{e,q}(1\otimes g),~\forall p,q\in Q,~h\in H_{pq},~g\in H_q,$$

$\beta^{-1}$ is almost left $H-$linear .

And we have
$$\beta^{\star}_{pq,q}\beta_{p,q}(h\otimes g)=\beta^{\star}_{pq,q}(hg_{(1,q)}\otimes g_{(2,q)})=(hg_{(1,q)})S_q(g_{(2,q)})\otimes g_{(3,q)}$$$$\stackrel{(\ref{2d})}=h\otimes g,~\forall p,q\in Q,~h\in H_p,~g\in H_q,$$
and
$$\beta_{pq^{-1},q}\beta^{\star}_{p,q}(h\otimes g)=\beta_{pq^{-1},q}(hS_q(g_{(1,q)})\otimes g_{(2,q)})=(hS_q(g_{(1,q)}))g_{(2,q)}\otimes g_{(3,q)}$$
$$\stackrel{(\ref{2d})}=h\otimes g,~\forall p,q\in
Q,~h\in H_p,~g\in H_q.$$

Analogous computation yields the result for $\gamma^{\star}$.

$(\Longleftarrow)$ We denote
$$\beta^{\star}_{e,q}(1\otimes g)\equiv g^{[1,q^{-1}]}\otimes g^{[2,q]},~\forall q\in Q,~g\in H_q,$$
$$\gamma^{\star}_{p,e}(h\otimes 1)\equiv h^{(1,p)}\otimes h^{(2,p^{-1})}, ~\forall p\in Q,~h\in H_p.$$

Since $\beta^{\star}$ and $\gamma^{\star}$ are almost $H-$linear, we have
$$\beta^{\star}_{p,q}(h\otimes g)=(h\otimes 1)\beta^{\star}_{e,q}(1\otimes g)=hg^{[1,q^{-1}]}\otimes g^{[2,q]},~\forall p,q\in Q,~h\in H_p,~g\in H_q,$$
$$\gamma^{\star}_{p,q}(h\otimes g)=\gamma^{\star}_{p,e}(h\otimes 1)(1\otimes g)=h^{(1,p)}\otimes h^{(2,p^{-1})}g,~\forall p,q\in Q,~h\in H_p,~g\in H_{q}.$$

Define two families of $k-$linear maps:
$$S=\{S_q:H_q\rightarrow H_{q^{-1}},S_q(g)=g^{[1,q^{-1}]}\epsilon_q(g^{[2,q]})\}_{q\in Q}$$
$$S^{\star}=\{S^{\star}_p:H_p\rightarrow H_{p^{-1}},S^{\star}_p(h)=\epsilon_p(h^{(1,p)})h^{(2,p^{-1})}\}_{p\in Q}.$$

Since for any $p\in Q,$ $\Delta_p$ is coassociative and $\beta$ is almost right $H-$colinear, by lemma 2.2, $\beta$ is right $H-$colinear, and hence $\beta^{-1}$ is right $H-$colinear. So
$$g_{(1,q)}^{[1,q^{-1}]}\otimes g_{(1,q)}^{[2,q]}\otimes g_{(2,q)}=\beta^{\star}_{e,q}(1\otimes g_{(1,q)})\otimes g_{(2,q)}=(I\otimes \Delta_q)\beta^{\star}_{e,q}(1\otimes g)=(I\otimes \Delta_q)(g^{[1,q^{-1}]}\otimes g^{[2,q]})$$
$$=g^{[1,q^{-1}]}\otimes g^{[2,q]}_{(1,q)}\otimes g^{[2,q]}_{(2,q)},~\forall q\in Q,~g\in H_q.$$

Applying $I\otimes \epsilon_q\otimes I$ to both sides of above equation, we have
$$g^{[1.q^{-1}]}\otimes g^{[2,q]}=g^{[1,q^{-1}]}\otimes \epsilon_q(g^{[2,q]}_{(1,q)})g^{[2,q]}_{(2,q)}=g^{[1,q^{-1}]}_{(1,q)}\epsilon_q(g^{[2,q]}_{(1,q)})\otimes g_{(2,q)}=S_q(g_{(1,q)})\otimes g_{(2,q)}.$$

So, we have
$$h\otimes g=\beta^{\star}_{pq,q}\beta_{p,q}(h\otimes g)=\beta^{\star}_{pq,q}(hg_{(1,q)}\otimes g_{(2,q)})
=(hg_{(1,q)}\otimes 1)\beta^{\star}_{e,q}(1\otimes g_{(2,q)})$$
$$=(hg_{(1,q)})S_q(g_{(2,q)})\otimes g_{(3,q)},~\forall p,q\in Q,~h\in H_{p},~g\in H_q,$$

and
$$h\otimes g=\beta_{pq^{-1},q}\beta^{\star}_{p,q}(h\otimes g)
=\beta_{pq^{-1},q}((h\otimes 1)\beta^{\star}_{e,q}(1\otimes g))
=\beta_{pq^{-1},q}(hS_q(g_{(1,q)})\otimes g_{(2,q)})$$
$$=(hS_q(g_{(1,q)}))g_{(2,q)}\otimes g_{(3,q)}~\forall p,q\in Q,~h\in H_{p},~g\in H_q,$$

so we get
$$(hg_{(1,q)})S_q(g_{(2,q)})\otimes g_{(3,q)}=h\otimes g=(hS_q(g_{(1,q)}))g_{(2,q)}\otimes g_{(3,q)}.$$

Applying $I\otimes \epsilon_q$ to the above equation, we have
$$(hg_{(1,q)})S_q(h_{(2,q)})=h\epsilon_q(g)=(hS_q(g_{(1,q)}))h_{(2,q)},~\forall p,q\in Q,~h\in H_p,~g\in H_q.$$

Following similar chain of arguments we can conclude that the the map $S^{\star}$ satisfy
$$S_p^{\star}(h_{(1,p)})(h_{(2,p)}g)=\epsilon_p(h)g=h_{(1,p)}(S_p^{\star}(h_{(2,p)})g),~\forall p,q\in Q,~h\in H_p,~g\in H_q.$$

Finally, we can compute
$$S_p^{\star}(h)=S_p^{\star}(h_{(1,p)})\epsilon_p(h_{(2,p)})=(S_p^{\star}(h_{(1,p)})h_{(2,p)})h_{(3,p)}$$
$$=\epsilon_p(h_{(1,p)})S_p(h_{(2,p)})=S_p(h),~\forall p\in Q,~h\in H_p.$$

Thus the maps $S$ and $S^{\star}$ coincide, $H$ is a $Q-$graded Hopf quasigroup.
\end{proof}

  \section{$Q-$graded quasimodules and $Q-$graded Hopf quasimodules}
  \def\theequation{4.\arabic{equation}}
  \setcounter{equation} {0} \hskip\parindent
 \begin{definition}
 Let $H$ be a $Q-$graded Hopf quasigroup , a $Q-$graded  $H-$quasimodule is a pair $(M,\varphi)$, where $M=\{M_q\}_{q\in Q}$ is a family of $k-$spaces and $\varphi=\{\varphi_{p,q}:H_p\otimes M_q\rightarrow M_{pq}\}_{p,q\in Q}$ ($Q$-graded action) is a family of $k-$maps such that

 $$1\cdot m=m,~\forall q\in Q,~m\in M_q,$$
\begin{equation}
 h_{(1,p)}\cdot (S_p(h_{(2,p)})\cdot m)=S_p(h_{(1,p)})\cdot (h_{(2,p)}\cdot m)=\epsilon_p(h)m,~\forall p,q\in Q,~h\in H_p,~m\in M_q. \label{4a}
\end{equation}
 A map of $Q-$graded  $H-$quasimodules is a family $f=\{f_q:M_q\rightarrow N_q\}_{q\in Q}$ of $k-$maps such that
 \begin{equation}
 f_q(h_{(1,p)}\cdot (S_p(h_{(2,p)})\cdot m))=h_{(1,p)}\cdot (S_p(h_{(2,p)})\cdot f_q(m)),~\forall p,q\in Q,~h\in H_p,~m\in M_q.\label{4b}
 \end{equation}
 \end{definition}

\begin{remark}
$H$ is a $Q-$graded Hopf quasimodule with the $Q-$graded multiplication as $Q-$graded action.
\end{remark}
 \begin{example}
 Let $Q$ be a quasigroup and $X=\{X_q\}_{q\in Q}$ is a set, then we can define a $Q-$graded  $kQ-$quasimodule on a family $kX=\{(kX_q)\}_{q\in Q}$ of $k-$spaces as follows:
 $$(k_1p)\cdot (k_2X_q)=k_1k_2X_{pq},~\forall p,q\in Q.$$
 \end{example}

 \begin{proposition}
 Suppose $M,$ N are two $Q-$graded  $H-$quasimodules, then $M\otimes N=\{M_q\otimes N_q\}_{q\in Q}$ is a $Q-$graded  $H-$quasimodule with the diagonal action:
 $$h\cdot (m\otimes n)=h_{(1,p)}\cdot m\otimes h_{(2,p)}\cdot n,~\forall p,q\in Q,~h\in H_p,~m\otimes n\in M_q\otimes N_q.$$
 \end{proposition}
 \begin{proof}
 $$\begin{aligned}
 &h_{(1,p)}\cdot (S_p(h_{(2,p)})\cdot (m\otimes n))\\
 &=h_{(1,p)}\cdot(S_p(h_{(2,p)})_{(1,p^{-1})}\cdot m\otimes S_p(h_{(2,p)})_{(2,p^{-1})}\cdot n)\\
 &\stackrel{(\ref{2g})}=h_{(1,p)}\cdot (S_p(h_{(3,p)})\cdot m\otimes S_p(h_{(2,p)})\cdot n)\\
 &=h_{(1,p)}\cdot (S_p(h_{(4,p)})\cdot m)\otimes h_{(2,p)}\cdot (S_p(h_{(3,p)})\cdot n)\\
 &\stackrel{(\ref{4a})}=\epsilon_p(h)(m\otimes n).\\
 \end{aligned}$$

 Similarly, we can proof
 $S_p(h_{(1,p)})\cdot (h_{(2,p)}\cdot (m\otimes n))=\epsilon_p(h)(m\otimes n),$ $M\otimes N$ with the diagonal action is a $Q-$graded  $H-$quasimodule.
 \end{proof}

\begin{definition}
Let $H$ be a $Q-$graded Hopf quasigroup, a $Q-$graded  $H-$Hopf quasimodule is a $Q-$graded  $H-$quasimodule $M=\{M_q,\varphi_{p,q}\}_{p,q\in Q}$ and $\forall q\in Q,$ $M_q$ is coassociative counitary left $H_q-$comodule $(M_q,\rho_q)$ such that (we denote $\rho_{q}(m)=m_{(-1,q)}\otimes m_{(0,q)}$)
\begin{equation}
(h\cdot m)_{(-1,pq)}\otimes (h\cdot m)_{(0,pq)}=h_{(1,p)}m_{(-1,q)}\otimes h_{(2,p)}\cdot m_{(0,q)}~\forall p,q\in Q,~h\in H_p,~m\in M_q. \label{4c}
\end{equation}

A map of $Q-$graded  $H-$Hopf quasimodules is a map of $Q-$graded  $H-$quasimodules $f=\{f_q:M_q\rightarrow N_q\}_{q\in Q}$ such that
\begin{equation}
m_{(-1,q)}\otimes f_q(m_{(0,q)})=f_q(m)_{(-1,q)}\otimes f_q(m)_{(0,q)},~\forall q\in Q,~m\in M_q. \label{4d}
\end{equation}
\end{definition}

\begin{example}
Let $Q$ be a quasigroup, we can define a $Q-$graded $kQ-$Hopf quasimodule on $Q-$graded $kQ-$quasimodule $kX=\{(kX_q)\}_{q\in Q}$ by
$$\rho_q(X_q)=q\otimes X_q,~\forall q\in Q.$$
\end{example}

\begin{lemma}
Let $M$ be a $Q-$graded  $H-Hopf$ quasimodule, then
\begin{equation}
 \rho_e(S_q(m_{(-1,q)})\cdot m_{(0,q)})=1\otimes S_q(m_{(-1,q)})\cdot m_{(0,q)},~\forall q\in Q,~m\in M_q. \label{4e}
\end{equation}
\begin{equation}
 \rho_q(h\cdot (S_q(m_{(-1,q)})\cdot m_{(0,q)}))=h_{(1,p)}\otimes h_{(2,p)}\cdot (S_q(m_{(-1,q)})\cdot m_{(0,q)}),~\forall p,q\in Q,~h\in H_p,~m\in M_q. \label{4f}
\end{equation}

\end{lemma}

\begin{proof}
$$\aligned
&\rho_e(S_q(m_{(-1,q)})\cdot m_{(0,q)})\\
&=(S_q(m_{(-1,q)})\cdot m_{(0,q)})_{(-1,e)}\otimes (S_q(m_{(-1,q)})\cdot m_{(0,q)})_{(0,e)}\\
&\stackrel{(\ref{4c})}=(S_q(m_{(-2,q)}))_{(1,q^{-1})}m_{(-1,q)}\otimes (S_q(m_{(-2,q)}))_{(2,q^{-1})}\cdot m_{(0,q)}\\
&\stackrel{(\ref{2g})}=S_q(m_{(-2,q)})m_{(-1,q)}\otimes S_q(m_{(-3,q)})\cdot m_{(0,q)}\\
&\stackrel{(\ref{2e})}=\epsilon_q(m_{(-1,q)})1\otimes S_q(m_{(-2,q)})\cdot m_{(0,q)}\\
&=1\otimes S_q(m_{(-1,q)})\cdot m_{(0,q)},\\
\endaligned$$

$$\aligned
&\rho_q(h\cdot (S_q(m_{(-1,q)})\cdot m_{(0,q)}))\\
&\stackrel{(\ref{4c})}=h_{(1,p)}(S_q(m_{(-1,q)})\cdot m_{(0,q)})_{(-1,q)}\otimes h_{(2,p)}\cdot (S_q(m_{(-1,q)})\cdot m_{(0,q)})_{(0,q)}\\
&\stackrel{(\ref{4e})}=h_{(1,p)}\otimes h_{(2,p)}\cdot (S_q(m_{(-1,q)})\cdot m_{(0,q)}).\\
\endaligned$$
\end{proof}

\begin{theorem}
Let $(M,\varphi,\rho)$ be a $Q-$graded  $H-$Hopf quasimodule, we replace the $Q-$graded action $\varphi$ by
$\vartriangleright=\{\vartriangleright_{p,q}:H_p\otimes M_q\rightarrow M_{pq},\vartriangleright_{p,q}(h\otimes m)=(hm_{(-2,q)})\cdot (S_q(m_{(-1,q)})\cdot m_{(0,q)})\}_{p,q\in Q} $ (we denote $\vartriangleright_{p,q}(h\otimes m)\equiv h\vartriangleright m$), then $(M,\vartriangleright,\rho)$ is a $Q-$graded  $H-$Hopf quasimodule.
\end{theorem}
\begin{proof}

Obviously,
$$\begin{aligned}
&1\vartriangleright m=m_{(-2,q)}\cdot (S_q(m_{(-1,q)})\cdot m_{(0,q)})
=\epsilon_q(m_{(-1,q)})m_{(0,q)}=m,~\forall q\in Q,~m\in M_q.\\
\end{aligned}$$

Verifying the conditon $(\ref{4a}): $ince
$$g\vartriangleright m=(gm_{(-2,r)})\cdot (S_r(m_{(-1,r)})\cdot m_{(0,r)}),~\forall q,r\in Q,~g\in H_q,~r\in M_r,$$

we have
$$\rho(g\vartriangleright m)\stackrel{(\ref{4f})}=(gm_{(-2,r)})_{(1,qr)}\otimes (gm_{(-2,r)})_{(2,qr)}\cdot (S_r(m_{(-1,r)})m_{(0,r)}),$$

and
$$(g\triangleright m)_{(-2,qr)}\otimes (g\triangleright m)_{(-1,qr)}\otimes (g\triangleright m)_{(0,qr)}$$
$$=(gm_{(-2,r)})_{(1,qr)}\otimes (gm_{(-2,r)})_{(2,qr)}\otimes (gm_{(-2,r)})_{(3,qr)}\cdot (S_r(m_{(-1,r)})\cdot m_{(0,r)}),$$

using the above equation,
$$\begin{aligned}
&h\triangleright(g\triangleright m)\\
&=(h(g\triangleright m)_{(-2,qr)})\cdot (S_{qr}((g\vartriangleright m)_{(-1,qr)})\cdot (g\vartriangleright m)_{(0,qr)})\\
&=(h(gm_{(-2,r)})_{(1,qr)})\cdot (S_{qr}((gm_{(-2,r)})_{(2,qr)})\cdot ((gm_{(-2,r)})_{(3,qr)}\cdot (S_r(m_{(-1,r)})\cdot m_{(0,r)})))\\
&=(h(gm_{(-2,r)}))\cdot (S_r(m_{(-1,r)})\cdot m_{(0,r)}),~\forall p,q,r\in Q,~h\in H_p,~g\in H_q,~m\in M_r,\\
\end{aligned}$$

replacing $h$ with $S_p(h_{(1,p)})$ and $g$ with $h_{(2,p)},$

$$\begin{aligned}
&S_p(h_{(1,p)})\vartriangleright (h_{(2,p)}\vartriangleright m)\\
&=(S_p(h_{(1,p)})(h_{(2,p)}m_{(-2,q)}))\cdot (S_q(m_{(-1,q)})\cdot m_{(0,q)})\\
&=\epsilon_p(h)m_{(-2,q)}\cdot (S_q(m_{(-1,q)})\cdot m_{(0,q)})\\
&=\epsilon_p(h)m,~\forall p,q\in Q,~h\in H_p,~m\in M_q.\\
\end{aligned}$$

Similary, we can get
$$h_{(1,p)}\vartriangleright (S_p(h_{(2,p)})\vartriangleright m)=\epsilon_p(h)m~,\forall p,q\in Q,~h\in H_p,~m\in M_q.$$

So, $(M,\varphi)$ is a $Q-$graded  $H-$quasimodule. Now we verify the condition $(\ref{4c})$ as follows:

$$\begin{aligned}
&\rho(h\vartriangleright m)=\rho((hm_{(-2,q)})\cdot (S_q(m_{(-1,q)})\cdot m_{(0,q)}))\\
&\stackrel{(\ref{4f})}=(hm_{(-2,q)})_{(1,pq)}\otimes (hm_{(-2,q)})_{(2,pq)}\cdot (S_q(m_{(-1,q)})\cdot m_{(0,q)})\\
&\stackrel{(\ref{2a})}=h_{(1,p)}m_{(-3,q)}\otimes (h_{(2,p)}m_{(-2,q)})\cdot (S_q(m_{(-1,q)})\cdot m_{(0,q)})\\
&=h_{(1,p)}m_{(-1,q)}\otimes h_{(2,p)}\vartriangleright m_{(0,q)}.\\
\end{aligned}$$

\end{proof}

\begin{definition}
Let $M$ be a $Q-$graded  $H-$Hopf quasimodule, the coinvariant of $M$ on $H$ is a vector space $$M^{coH}=\{m\in M_e\| \rho_e(m)=1\otimes m\}.$$
\end{definition}

\begin{lemma}
Let $M=\{M_q\}_{q\in Q}$ be a $Q-$graded  $H-$Hopf quasimodule, then $H\otimes M^{coH}=\{H_q\otimes M^{coH}\}_{q\in Q}$ is a $Q-$graded  $H-$Hopf quasimodule.
\end{lemma}

\begin{proof}
Defining the $Q-$graded  $H-$module action $\varphi$ on $H\otimes M^{coH}$:
$$h\cdot(g\otimes m)=hg\otimes m,~\forall p,q\in Q,~h\in H_p,~g\otimes m\in H_q\otimes M^{coH}.$$

Since
$$1\cdot (g\otimes m)=g\otimes m,~\forall q\in Q,~g\otimes  m\in H_q\otimes M^{coH},$$

and

$$\begin{aligned}
&S_p(h_{(1,p)})\cdot (h_{(2,p)}\cdot (g\otimes m))\\
&=S_p(h_{(1,p)})(h_{(2,p)}g)\otimes m\\
&\stackrel{(\ref{2c})}=\epsilon_p(h)(g\otimes m)\\
&\stackrel{(\ref{2c})}=h_{(1,p)}(S_p(h_{(2,p)})g)\otimes m\\
&=h_{(1,p)}\cdot (S_p(h_{(2,p)})\cdot (g\otimes m)),\\
&~\forall p,q\in Q,~h\in H_p,~g\otimes m\in H_q\otimes M^{coH},\\
\end{aligned}$$

$(H\otimes M^{coH},\varphi)$ is a $Q-$graded  $H-$quasimodule.

For any $q\in Q,$ we define $H_q-$comodule coaction on $H_q\otimes M^{coH}$:
$$\rho_q(g\otimes m)=g_{(1,q)}\otimes g_{(2,q)}\otimes m,~\forall q\in Q,~g\otimes m\in H_q\otimes M^{coH}.$$

Since
$$\aligned
&(\Delta_q\otimes I\otimes I)(\rho_q(g\otimes m))\\
&=(\Delta_q\otimes I\otimes I)(g_{(1,q)}\otimes g_{(2,q)}\otimes m)\\
&=g_{(1,q)}\otimes g_{(2,q)}\otimes g_{(3,q)}\otimes m\\
&=(I\otimes \rho_q)(g_{(1,q)}\otimes g_{(2,q)}\otimes m)\\
&=(I\otimes \rho_q)\rho_q(g\otimes m),~
\forall q\in Q,~g\otimes m\in H_q\otimes M^{coH},\\
\endaligned$$

and
$$\aligned
&(\epsilon_q\otimes I\otimes I)\rho_q(g\otimes m)=(\epsilon_q\otimes I\otimes I)(g_{(1,q)}\otimes g_{(2,q)}\otimes m)=g\otimes m,~\forall q\in Q,~g\otimes m\in H_q\otimes M^{coH},\\
\endaligned$$

$(H_q\otimes M^{coH},\rho_q)$ is a coassociative counitary  $H_q-$comodule.

Now we verify the condition $(\ref{4c})$ for $H\otimes M^{coH},$
$$\aligned
&(h\cdot (g\otimes m))_{(-1,pq)}\otimes (h\cdot (g\otimes m))_{(0,pq)}\\
&=(hg\otimes m)_{(-1,pq)}\otimes (hg\otimes m)_{(0,pq)}\\
&=(hg)_{(-1,pq)}\otimes (hg)_{(2,pq)}\otimes m\\
&\stackrel{(\ref{2a})}=h_{(1,p)}g_{(1,p)}\otimes h_{(2,p)}g_{(2,p)}\otimes m\\
&=h_{(1,p)}(g\otimes m)_{(-1,q)}\otimes h_{(2,p)}\cdot (g\otimes m)_{(0,q)},\\
&~\forall p,q\in Q,~h\in H_p,~g\otimes m\in H_q\otimes M^{coH}.\\
\endaligned$$

So, $H\otimes M^{coH}$ is a $Q-$graded  $H-$Hopf quasimodule.
\end{proof}

\begin{theorem}
Let $M=\{M_q\}_{q\in Q}$ be a $Q-$graded  $H-$Hopf quasimodule, then (as $Q-$graded  $H-$Hopf quasimodule)
 $$H\otimes M^{coH}\cong M.$$
\end{theorem}
\begin{proof}
For any $q\in Q,$ we define
$$\sigma_q:H_q\otimes M^{coH}\rightarrow M_q,g\otimes m\rightarrow g\cdot m,~\forall q\in Q,~g\otimes M\in H_q\otimes M^{coH}.$$

Verify $\sigma=\{\sigma_q\}_{q\in Q}$ satisfy the condition (\ref{4b}):
$$\aligned
&\sigma_q(h_{(1,p)}\cdot (S_p(h_{(2,p)})\cdot (g\otimes m)))\\
&=\sigma_q(h_{(1,p)}(S_p(h_{(2,p)})g)\otimes m)\\
&\stackrel{(\ref{4a})}=\sigma_q(\epsilon_p(h)(g\otimes m))\\
&=\epsilon_p(h)g\cdot m\\
&\stackrel{(\ref{4a})}=h_{(1,p)}\cdot (S_p(h_{(2.p)})\cdot (g\cdot m))\\
&=h_{(1,p)}\cdot (S_p(h_{(2,p)})\cdot \sigma_q(g\otimes m)),\\
&\forall p,q\in Q,~h\in H_p,~g\otimes m\in H_q\otimes M^{coH}.\\
\endaligned$$

Verify $\sigma=\{\sigma_q\}_{q\in Q}$ sastisfy the condtion (\ref{4d})
$$\aligned
&(g\otimes m)_{(-1,q)}\otimes \sigma_q((g\otimes m)_{(0,q)})\\
&=g_{(1,q)}\otimes \sigma_q(g_{(2,q)}\otimes m)\\
&=g_{(1,q)}\otimes g_{(2,q)}\cdot m\\
&\stackrel{(\ref{4c})}=(g\cdot m)_{(-1,q)}\otimes (g\cdot m)_{(0,q)}\\
&=\sigma_q(g\otimes m)_{(-1,q)}\otimes \sigma_q(g\otimes m)_{(0,q)},\\
&\forall q\in Q,~g\otimes m\in H_q\otimes M^{coH}.\\
\endaligned$$

So, $\sigma$ is a map of $Q-$graded  $H-$quasimodules.

For any $q\in Q,$ we define
$$\sigma_q^{-1}:M_q\rightarrow H_q\otimes M^{coH},~m\rightarrow m_{(-2,q)}\otimes S_q(m_{(-1,q)}\cdot m_{(0,q)}),$$

noticing by (\ref{4e}), $\sigma_q^{-1}(m)\subset H_q\otimes M^{coH},~\forall q\in Q,~m\in M_q,$ so $\sigma^{-1}$ is wel defined.

Since
$$\aligned
&\sigma_q(\sigma_q^{-1}(m))\\
&=m_{(-2,q)}\cdot (S_q(m_{(-1,q)})\cdot m_{(0,q)})\\
&\stackrel{(\ref{4a})}=\epsilon_q(m_{(-1,q)})m_{(0,q)}\\
&=m,~\forall q\in Q,~m\in M_q,\\
\endaligned$$

and

$$\aligned
&\sigma_{q^{-1}}\sigma_q(g\otimes m)\\
&=(g\cdot m)_{(-2,q)}\otimes S_q((g\cdot m)_{(-1,q)})\cdot (g\cdot m)_{(0,q)}\\
&\stackrel{(\ref{4c})}=g_{(1,q)}m_{(-2,q)}\otimes S_q(g_{(2,q)}m_{(-1,q)})\cdot (g_{(3,q)}\cdot m_{(0,q)})\\
&=g_{(1,q)}\otimes S_q(g_{(2,q)})\cdot (g_{(3,q)}\cdot m)\\
&\stackrel{(\ref{4a})}=g\otimes m,~\forall q\in Q,~g\otimes m\in H_q\otimes M^{coH},\\
\endaligned$$

$\sigma $ is an isomorphism of $Q-$graded  $H-$Hopf quasimodules, then $M\cong H\otimes M^{coH}$.

\end{proof}

\section{$Q-$graded Long dimodules}
\def\theequation{5.\arabic{equation}}
\setcounter{equation} {0} \hskip\parindent

\begin{definition}
Let $H$ be a $Q-$graded Hopf quasigroups, a $Q-$graded Long $H-$dimodule is a triple $(M,\varphi,\rho),$ where $(M,\varphi)$ is a $Q-$graded left $H-$quasimodule, $(M,\rho)=\{(M_q,\rho_{q,e}):\rho_{q,e}:M_q\rightarrow M_q\otimes H_e,\rho_{q,e}(m)\equiv m_{(o,q)\otimes m_{(1,e)}}\}_{q\in Q}$ is a family of coassociative counitary right $H_e-$comodules such that
\begin{equation}
(h\cdot m)_{(0,pq)}\otimes (h\cdot m)_{(1,e)}=h\cdot m_{(0,q)}\otimes m_{(1,e)},~\forall p,q\in Q,~h\in H_p,~m\in M_q. \label{5a}
\end{equation}
\end{definition}

\begin{example}
Let $Q$ be a quasigroup , we can define a $Q-$graded Long $kQ-$dimodule on $Q-$graded $kQ-$quasimodule $kX$ by
$$\rho_{q,e}(X_q)=X_q\otimes e,~\forall q\in Q.$$
\end{example}

\begin{lemma}
Let $(M,\varphi,\rho)$ be a $Q-$graded Long $H-$dimodule, then we have
\begin{equation}
\rho_{q,e}(S_e(m_{(1,e)})\cdot m_{(0,q)})=S_e(m_{(1,e)(2,e)})\cdot m_{(0,q)}\otimes m_{(1,e)(1,e)},~\forall q\in Q,~m\in M_q, \label{5b}
\end{equation}

and
\begin{equation}
\rho_{pq,e}(h\cdot (S_e(m_{(1,e)})\cdot m_{(0,q)}))=h\cdot (S_e(m_{(1,e)(2,e)})\cdot m_{(0,q)})\otimes m_{(1,e)(1,e)},~\forall p,q\in Q,~h\in H_p,~m\in M_q. \label{5c}
\end{equation}
\end{lemma}

\begin{proof}

$$\begin{aligned}
&\rho_{q,e}(S_e(m_{(1,e)})\cdot m_{(0,q)})\\&
=(S_e(m_{(1,e)})\cdot m_{(0,q)})_{(0,q)} \otimes (S_e(m_{(1,e)})\cdot m_{(0,q)})_{(1,e)}\\
&\stackrel{(\ref{5a})}=S_e(m_{(2,e)})\cdot m_{(0,q)}\otimes m_{(1,e)}\\
&=S_e(m_{(1,e)(2,e)})\cdot m_{(0,q)}\otimes m_{(1,e)(1,e)},\\
\end{aligned}$$
and
$$\begin{aligned}
&\rho_{pq,e}(h\cdot (S_e(m_{(1,e)})\cdot m_{(0,q)}))\\
&=(h\cdot (S_e(m_{(1,e)})\cdot m_{(0,q)}))_{(0,pq)}\otimes (h\cdot (S_e(m_{(1,e)})\cdot m_{(0,q)}))_{(1,e)}\\
&\stackrel{(\ref{5a})}=h\cdot (S_e(m_{(1,e)})\cdot m_{(0,q)})_{(0,q)}\otimes (S_e(m_{(1,e)})\cdot m_{(0,q)})_{(1,e)}\\
&\stackrel{(\ref{5b})}=h\cdot S_e(m_{(1,e)(2,e)})\cdot m_{(0,q)}\otimes m_{(1,e)(1,e)}.\\
\end{aligned}$$
\end{proof}

\begin{proposition}
Let $M$ be a $Q-$graded $H-$quasimodule, then $M\otimes H_e=\{M_q\otimes H_e\}_{q\in Q}$ is a $Q-$graded Long $H-$dimodule by :
 $$h\cdot (m\otimes g)=(h\cdot m)\otimes g,~\forall p,q\in Q,~h\in H_p,~m\in M_q,g\in H_e, $$
 $$\rho_{q,e}(m\otimes g)=m\otimes g_{(1,e)}\otimes g_{(2,e)},~\forall q\in Q,~m\in M_q,~g\in H_e.$$
\end{proposition}

\begin{proof}
Since
$$1\cdot (m\otimes g)=m\otimes g,$$

and
$$\aligned
&S_p(h_{(1,p)})\cdot(h_{(2,p)}\cdot (m\otimes g))\\
&=S_p(h_{(1,p)})\cdot(h_{(2,p)}\cdot m)\otimes g\\
&=\epsilon_p(h)(m\otimes g)\\
&=h_{(1,p)}\cdot (S_p(h_{(2,p)})\cdot m)\otimes g\\
&=h_{(1,p)}\cdot (S_p(h_{(2,p)}\cdot (m\otimes g)),\\
\endaligned$$

therefore $M\otimes H_e$ is a $Q-$graded $H-$quasimodule. By the coassociativity of $\Delta_e$ and the counitality of $\epsilon_e$ , it is easy to check $M\otimes H_e$ is a right $H_e-$comodule. Now we proof the condition $(\ref{5a})$ as follows:
$$\aligned
&\rho_{pq,e}(h\cdot (m\otimes g)\\
&=\rho_{pq,e}(h\cdot m\otimes g)\\
&=(h\cdot m\otimes g)_{(0,pq)}\otimes (h\cdot m\otimes g)_{(1,e)}\\
&=h\cdot m\otimes g_{(1,e)}\otimes g_{(2,e)}\\
&=h\cdot (m\otimes g)_{(0,q)}\otimes (m\otimes g)_{(1,e)},
\endaligned$$

so, $M\otimes H_e$ is a $Q-$graded Long $H-$dimodule.

\end{proof}

\begin{proposition}
Let $H$ be a $Q-$graded Hopf quasigroup and $M$ is a right $H_e-$comodule, then $H\otimes M=\{H_q\otimes M\}_{q\in Q}$ is a $Q-$graded Long $H-$dimodule by
$$h\cdot (g\otimes m)=hg\otimes m,~\forall p,q\in Q,~h\in H_p,~g\in H_q,~m\in M.$$
$$\rho(g\otimes m)=g\otimes m_0\otimes m_1,~\forall q\in Q,~g\in H_q,~m\in M.$$
\end{proposition}

\begin{proof}
Since
$$\begin{aligned}
&h_{(1,p)}\cdot (S_p(h_{(2,p)})\cdot (g\otimes m))\\&
=h_{(1,p)}\cdot (S_p(h_{(2,p)})\cdot g)\otimes m\\
&=\epsilon_p(h)(g\otimes m)\\
&=S_p(h_{(1,p)})\cdot (h_{(2,p)}\cdot g)\otimes m)\\\
&=S_p(h_{(1,p)})\cdot (h_{(2,p)}\cdot (g\otimes m)),\\
\end{aligned}$$

$H\otimes M$ is a $Q-$graded left $H-$quasimodule.

Verify $H\otimes M$ is a right $H_e-$comodule is easy and be left to reader, we verify condtion $(\ref{5a})$ as follows:
$$\begin{aligned}
&(h\cdot (g\otimes m))_{(0,pq)}\otimes (h\cdot (g\otimes m))_{(1,e)}\\
&=(hg\otimes m)_{(0,pq)}\otimes (hg\otimes m)_{(1,e)}\\
&=hg\otimes m_0\otimes m_1\\
&=h\cdot (g\otimes m_0)\otimes m_1\\
&=h\cdot (g\otimes m)_{(0,q)}\otimes (g\otimes m)_{(1,e)}.\\
\end{aligned}$$

Then, $H\otimes M$ is a $Q-$graded Long $H-$dimodule.

\end{proof}

\begin{corollary}
Let $(M,\varphi)$ be a $Q-$graded left $H-$quasimodule, then $(M,\varphi,\rho)$, then $M$ is a $Q-$graded Long $H-$dimodule by $$\rho_{q,e}(m)=m\otimes 1.$$
\end{corollary}

\begin{proposition}
Suppose $(M,\varphi_M,\rho_M)$  and $(N,\varphi_N,\rho_N)$ are two $Q-$graded Long $H-$dimodules, then $\{M_q\otimes N_q\}_{q\in Q}$ is a $Q-$graded Long $H-$dimodule by
$$h\cdot (m\otimes n)=h_{(1,p)}\cdot m\otimes h_{(2,p)}\cdot n,~\forall p,q\in Q,~h\in H_p,~m\otimes n\in M_q\otimes N_q,$$
$$\rho_{q,e}(m\otimes n)=m_{(0,q)}\otimes n_{(0,q)}\otimes m_{(1,e)}n_{(1,e)},~\forall q\in Q,~m\otimes n\in M_q\otimes N_q.$$
\end{proposition}

\begin{proof}
Since
$$1\cdot (m\otimes n)=1\cdot m\otimes 1\cdot n=m\otimes n.$$

and
$$\begin{aligned}
&h_{(1,p)}\cdot (S_p(h_{(2,p)})\cdot (m\otimes n))\\
&=h_{(1,p)}\cdot (S_p(h_{(3,p)})\cdot m\otimes S_p(h_{(2,p)})\cdot m)\\
&=h_{(1,p)}\cdot (S_p(h_{(4,p)})\cdot m)\otimes h_{(2,p)}\cdot (S_p(h_{(3,p)}\cdot n)\\
&=h_{(1,p)}\cdot S_p(h_{(3,p)}\cdot m)\otimes \epsilon_p(h_{(2,p)})n\\
&=\epsilon_p(h)(m\otimes n),\\
\end{aligned}$$

similarly, we can get $S_p(h_{(1,p)})\cdot (h_{(2,p)}\cdot (m\otimes n))=\epsilon_p(h)(m\otimes n),$ so $\{M_q\otimes N_q\}_{q\in Q}$ is a $Q-$graded $H-$quasimodule.

It is easy to check $\{M_p\otimes N_p\}_{q\in Q}$ is a family of right $H_e-$comodules, now we verify the condition
(\ref{5a}):
$$\begin{aligned}
&\rho_{pq,e}(h\cdot (m\otimes n))=\rho_{pq,e}(h_{(1,p)}\cdot m\otimes h_{(2,p)}\cdot n)\\
&=(h_{(1,p)}\cdot m)_{(0,pq)}\otimes (h_{(2,p)}\cdot n)_{(0,pq)}\otimes (h_{(1,p)}\cdot m)_{(1,e)}(h_{(2,p)}\cdot n)_{(1,e)}\\
&\stackrel{(\ref{5a})}=h_{(1,p)}\cdot m_{(0,q)}\otimes h_{(2,p)}\cdot n_{(0,q)}\otimes m_{(1,e)}n_{(1,e)}\\
&=h\cdot (m\otimes n)_{(0,q)} \otimes (m\otimes n)_{(1,e)},\\
\end{aligned}$$

therefore $\{M_q\otimes N_q\}_{q\in Q}$ is a $Q-$graded Long $H-$dimodule.
\end{proof}

\begin{proposition}
Let $(M,\varphi,\rho)$ be a Long $H-$dimodule. Then, for any $q\in Q$, the map
$$R_q:M_q\otimes M_q\rightarrow M_q\otimes M_q,~R_q(m\otimes n)=n_{(1,e)}\cdot m\otimes n_{(0,q)},$$
is a solution of Long-equation for $M_q$.
\end{proposition}

\begin{proof}
$$\begin{aligned}
&R_q^{12}R_q^{23}(l\otimes m\otimes n)\\
&=R^{12}_q(l\otimes n_{(1,e)}\cdot m\otimes n_{(0,q)})\\
&=(n_{(1,e)}\cdot m)_{(1,e)}l\otimes (n_{(1,e)}\cdot m)_{(0,q)}\otimes n_{(0,q)}\\
&\stackrel{(\ref{5a})}=m_{(1,e)}\cdot l\otimes n_{(1,e)}\cdot m_{(0,q)}\otimes n_{(0,q)}\\
&=R_{q}^{23}(m_{(1,e)}\cdot l\otimes m_{(0,q)}\otimes n)\\
&=R^{23}_qR^{12}_q(l\otimes m\otimes n),~\forall q\in Q,~l,m,n\in M_q.\\
\end{aligned}$$

we have $R_q^{12}R_q^{23}=R_{q}^{23}R_{q}^{12}~,\forall q\in Q.$

\end{proof}

By a $Q-$graded Long $H-$dimodule, we can construct a family of solution to Long-equation.

\section{The smash product for $Q$-graded Hopf quasigroups }
\def\theequation{6.\arabic{equation}}
\setcounter{equation} {0} \hskip\parindent
\begin{definition}
Let $H$ be a $Q-$graded Hopf quasigroup, a $H-$quasimodule is a pair $(M,\varphi),$ where $M$ is a $k-$space, $\varphi=\{\varphi_{p}:H_p\otimes M\rightarrow M\}_{p\in Q}$ is a family of $k-$maps such that,
$$1\cdot m=m,~\forall m\in M,$$
\begin{equation}
S_p(h_{(1,p)})\cdot( h_{(2,p)}\cdot m )=\epsilon_p(h)m=h_{(1,p)}\cdot (S_p(h_{(2,p)})\cdot m)~,\forall p\in Q,~h\in H_p,~m\in M. \label{6a}
\end{equation}
\end{definition}

\begin{definition}
Let $H$ be a $Q-$graded Hopf quasigroup, an (not necessarily associative) algebra $A$ is called a $H-$quasimodule algebra if $A$ is a $H-$quasimodule and
\begin{equation}
(h_{(1,p)}\cdot a )(h_{(2,p)}\cdot b)=h\cdot (ab),~\forall p\in Q,~h\in H_p,~a,b\in A, \label{6b}
\end{equation}
\begin{equation}
h\cdot 1= \epsilon_p(h)1,~\forall p\in Q,~h\in H_p. \label{6c}
\end{equation}

A coalgebra $C$ is a $H-$quasimodule coalgebra if $C$ is a $H-$quasimodule and
\begin{equation}
\Delta(h\cdot c)=h_{(1,p)}\cdot c_1\otimes h_{(2,p)}\cdot c_{2},~\forall p\in Q,~h\in H_p,~c\in C,\label{6d}
\end{equation}
\begin{equation}
~\epsilon(h\cdot c)=\epsilon_p(h)\epsilon(c),~\forall p\in Q,~h\in H_p,~c\in C.\label{6e}
\end{equation}

A Hopf quasigroup $A$ is called a $H-$quasimodule Hopf quasigroup if it is both a $H-$quasimodule algebra and a $H-$quasimodule coalgebra.
\end{definition}
\begin{lemma}
Let $H$ is a $Q-$graded Hopf quasigroup, $A$ is a $H-$quasimodule Hopf quasigroup, then
\begin{equation}
h\cdot S(a)=S(h\cdot a),~\forall p\in Q,~h\in H_p,~a\in A.\label{6f}
\end{equation}
\end{lemma}
\begin{proof}
Since
$$(h_{(1,p)}\cdot a_1)(h_{(2,p)}\cdot S(a_2))\stackrel{(\ref{6a})}=h\cdot (a_1S(a_2))\stackrel{(\ref{6c})}=\epsilon_p(h)\epsilon(a)1,$$

and
$$(h_{(1,p)}\cdot a_1)S(h_{(2,p)}\cdot a_2)=(h\cdot a)_1S((h\cdot a)_2)=\epsilon(h\cdot a)1\stackrel{(\ref{6e})}=\epsilon_p(h)\epsilon(a)1,$$

we have
$$(h_{(1,p)}\cdot a_1)(h_{(2,p)}\cdot S(a_2))=(h_{(1,p)}\cdot a_1)S(h_{(2,p)}\cdot a_2),$$

repalcing $h$ with $h_{(2,p)}$ and $a$ with $a_2,$ and multiplying by $S(h_{(1,p)}\cdot a_1)$ on the left, we have
\begin{equation}
S(h_{(1,p)}\cdot a_1)((h_{(2,p)}\cdot a_2)S(h_{(3,p)}\cdot a_3))=S(h_{(1,p)}\cdot a_1)((h_{(2,p)}\cdot a_2)(h_{(3,p)}\cdot S(a_3))), \label{6g}
\end{equation}

then we get
$$\begin{aligned}
&S(h\cdot a)=\epsilon(h_{(1,p)}\cdot a_1)S(h_{(2,p)}\cdot a_2)\\
&\stackrel{(\ref{2c})}=S((h_{(1,p)}\cdot a_1)_1)((h_{(1,p)}\cdot a_1)_2S(h_{(2,p)}\cdot a_2))\\
&\stackrel{(\ref{2a})}=S(h_{(1,p)}\cdot a_1)((h_{(2,p)}\cdot a_2)S(h_{(3,p)}\cdot a_3))\\
&\stackrel{(\ref{6f})}=S(h_{(1,p)}\cdot a_1)((h_{(2,p)}\cdot a_2)(h_{(3,p)}\cdot S(a_3)))\\
&\stackrel{(\ref{2a})}=S((h_{(1,p)}\cdot a_1)_1)((h_{(1,p)}\cdot a_1)_2(h_{(2,p)}\cdot S(a_2)))\\
&\stackrel{(\ref{2c})}=\epsilon(h_{(1,p)}\cdot a_1)(h_{(2,p)}\cdot S(a_2))\\
&=\epsilon_p(h_{(1,p)})\epsilon(a_1)(h_{(2,p)}\cdot S(a_2))=h\cdot S(a).
\end{aligned}$$
\end{proof}

\begin{theorem}
 Suppose $H$ be a $Q-$graded Hopf quasigroup and $A$ is a $H-$quasimodule Hopf quasigroup such that
\begin{equation}
h_{(1,p)}\otimes h_{(2,p)}\cdot a=h_{(2,p)}\otimes h_{(1,p)}\cdot a,~\forall p\in Q,~h\in H_p,~a\in A, \label{6h}
\end{equation}
then the following statements are equivalent:

(1) There is a smash product $Q-$graded Hopf quasigroup $A\rtimes H$ built on $A\otimes H=\{A\otimes H_p\}_{p\in Q}$ with

(smash product)
\begin{equation}(a\rtimes h)(b\rtimes g)=a(h_{(1,p)}\cdot b)\otimes h_{(2,p)}g,~\forall p,q\in Q,~h\in H_p,~g\i n H_q,~a,b\in A. \label{6i}
\end{equation}

(unit)
$$1_{A\rtimes H}=1\otimes 1_,$$

(Coproduct)
\begin{equation}
 \Delta^*_p(a\otimes h)=(a_1\otimes h_{(1,p)})\otimes (a_2\otimes h_{(2,p)}),
\end{equation} \label{6j}

(counit)
\begin{equation}
\epsilon^*_p(a\otimes h)=\epsilon(a)\epsilon_p(h), \label{6k}
\end{equation}

(antipode)
\begin{equation}
S^*_p(a\otimes h)=S_p(h_{(2,p)}\cdot S(a))\otimes S_p(h_{(1,p)}), \label{6l}
\end{equation}

(2)
\begin{equation}
g\cdot(S_p(h)\cdot a)=(gS_p(h))\cdot a,~\forall p,q\in Q,~h\in H_p,~g\in H_q,~a\in A. \label{6m}
\end{equation}
\end{theorem}

\begin{proof}
(1)$\Longrightarrow(2)$
$$\aligned
&\epsilon(a)\epsilon_p(h)b\otimes g\\
&\stackrel{(\ref{6i})}=(b\otimes g)\epsilon_p^{\star}(a\otimes h)\\
&\stackrel{(\ref{2d})}=((b\rtimes g)S_p^{\star}((a\rtimes h)_{(1,p)}))(a\rtimes h)_{(2,p)}\\
&\stackrel{(\ref{6j})}=((b\rtimes g)S_p^{\star}(a_1\rtimes h_{(1,p)}))(a_2\rtimes h_{(2,p)})\\
&\stackrel{(\ref{6l})}=((b\rtimes g)(S_p(h_{(2,p)}\cdot S(a_1))\rtimes S_p(h_{(1,p)})))(a_2\rtimes h_{(3,p)})\\
&\stackrel{(\ref{2g})}=((b\rtimes g)(S_p(h_{(1,p)})_{(1,p^{-1})}\cdot S(a_1)\rtimes S_p(h_{(1,p)})_{(2,p^{-1})}))(a_2\rtimes h_{(2,p)})\\
&\stackrel{(\ref{6i})}=(b(g_{(1,q)}\cdot (S_p(h_{(1,p)})_{(1,p^{-1})}\cdot S(a_1)))\rtimes g_{(2,q)}S_p(h_{(1,p)})_{(2,p^{-1})})(a_2\rtimes h_{(3,p)})\\
&\stackrel{(\ref{6i})}=(b(g_{(1,q)}\cdot(S(h_{(1,p)})_{(1,p^{-1})}\cdot S(a_1)))))((g_{(2,q)}S_p(h_{(1,p)})_{(2,p^{-1})})\cdot a_2)\otimes (g_{(3,q)}S(h_{(1,p)})_{(3,p^{-1})})h_{(2,p)}\\
&\stackrel{(\ref{6f})}=(bS(g_{(1,q)}\cdot (S_p(h_{(1,p)})_{(1,p^{-1})}\cdot a_1)))((g_{(2,q)}S_p(h_{(1,p)})_{(2,p^{-1})})\cdot a_2)\otimes (g_{(3,q)}S_p(h_{(1,p)})_{(3,p^{-1})})h_{(2,p)}.\\
\endaligned$$
Applying $I\otimes \epsilon_q$ to both sides of this equation, set $b=g_{(1,q)}\cdot (S_p(h)_{(1,p^{-1})}\cdot a_1)$ and replace $a, ~g,~h$ by $a_2,~g_{(2,q)},~h_{(2,q)},$ we obtain
$$\begin{aligned}
&g\cdot (S_p(h)\cdot a)\\
&=((g_{(1,q)}\cdot (S_p(h)_{(1,p^{-1})}\cdot a_1))S(g_{(2,q)}\cdot (S_p(h)_{(2,p^{-1})}\cdot a_2)))((g_{(3,q)}S_p(h)_{(3,p^{-1})})\cdot a_3)\\
&=((g_{(1,q)}\cdot (S_p(h)_{(1,p^{-1})}\cdot a_1)_1)S(g_{(2,q)}\cdot (S_p(h)_{(1,p^{-1})}\cdot a_1)_2))((g_{(3,q)}S_p(h)_{(3,p^{-1})})\cdot a_3)\\
&=((g_{(1,q)}\cdot (S_p(h)_{(1,p^{-1})}\cdot a_1))_1S((g_{(1,q)}\cdot (S_p(h)_{(1,p^{-1})}\cdot a_1))_2)((g_{(2,q)}S_p(h)_{(2,p^{-1})})\cdot a_2)\\
&=\epsilon^{\star}(g_{(1,q)}\cdot (S_p(h)_{(1,p^{-1})}\cdot a_1))((g_{(2,q)}S_p(h)_{(2,p^{-1})})\cdot a_2)\\
&=\epsilon_q(g_{(1,q)})\epsilon_{p^{-1}}(S_p(h)_{(1,p^{-1})})\epsilon(a_1)((g_{(2,q)}S_p(h)_{(2,p^{-1})})\cdot a_2)\\
&=(gS_p(h))\cdot a.\\
\end{aligned}$$

\end{proof}

(2)$\Longrightarrow (1)$

It is easy to check $A\rtimes H$ is a $Q-$graded algebra and $\forall p\in Q,~A\rtimes H_p$ is a coassocitive counitary coalgebra and $\mu$ is a coalgebra map,

Since
$$\begin{aligned}
&((a\rtimes h)(b\rtimes g))_{(1,pq)}\otimes ((a\rtimes h)(b\rtimes g))_{(2,pq)}=\Delta^*_{pq}((a\rtimes h)(b\rtimes g))\\
&\stackrel{(\ref{6i})}=\Delta^*_{pq}(a(h_{(1,p)}\cdot b)\otimes h_{(2,p)}g)\\
&\stackrel{(\ref{6j})}=(a(h_{(1,p)}\cdot b))_1\otimes (h_{(2,p)}g)_{(1,pq)}\otimes (a(h_{(1,p)}\cdot b))_2\otimes (h_{(2,p)}g)_{(2,pq)}\\
&=(a_1(h_{(1,p)}\cdot b)_1)\otimes (h_{(2,p)}g)_{(1,pq)}\otimes (a_2(h_{(1,p)}\cdot b)_2)\otimes (h_{(2,p)}g)_{(2,pq)}\\
&\stackrel{(\ref{6d})}=(a_1(h_{(1,p)}\cdot b_1))\otimes (h_{(3,p)}g)_{(1,pq)}\otimes (a_2(h_{(2,p)}\cdot b_2))\otimes (h_{(4,p)}g)_{(2,pq)}\\
&\stackrel{(\ref{2a})}=(a_1(h_{(1,p)}\cdot b_1))\otimes (h_{(3,p)}g_{(1,q)})\otimes (a_2(h_{(2,p)}\cdot b_2))\otimes (h_{(4,p)}g_{(2,q)})\\
&\stackrel{(\ref{6h})}=(a_1(h_{(1,p)}\cdot b_1))\otimes (h_{(2,p)}g_{(1,q)})\otimes (a_2(h_{(3,p)}\cdot b_2))\otimes (h_{(4,p)}g_{(2,q)})\\
&\stackrel{(\ref{6i})}=(a_1\rtimes h_{(1,p)})(b_1\rtimes g_{(1,q)})\rtimes (a_2\rtimes h_{(2,p)})(b_2\rtimes g_{(2,q)})\\
&\stackrel{(\ref{6j})}=(a\rtimes h)_{(1,p)}(b\rtimes g)_{(1,q)}\rtimes (a\rtimes h)_{(2,p)}(b\rtimes g)_{(2,q)},\\
\end{aligned}$$

and

$$\aligned
&\epsilon_{pq}^{\star}((a\rtimes h)(b\rtimes g))\\
&=\epsilon_{pq}^{\star}(a(h_{(1,p)}\cdot b)\otimes h_{(2,p)}g)\\
&=\epsilon(a(h_{(1,p)}\cdot b))\epsilon_{pq}(h_{(2,p)}g)\\
&=\epsilon(a)\epsilon_p(h_{(1,p)})\epsilon(b)\epsilon_p(h_{(2,p)})\epsilon_q(g)\\
&=\epsilon(a)\epsilon_p(b)\epsilon(b)\epsilon_q(g)\\
&=\epsilon^{\star}_p(a\otimes b)\epsilon^{\star}_q(b\otimes g),\\
\endaligned$$

the condition ($\ref{2a}$) is verified.

Since

$$\begin{aligned}&S^*_{p}((a\rtimes h)_{(1,
p)})((a\rtimes h)_{(2,p)}(b\rtimes g))\\
&=S^*_{p}(a_1\rtimes h_{(1,p)})((a_2\rtimes h_{(2,p)})(b\rtimes g))\\
&=(S_p(h_{(1,p)})_{(1,p^{-1})}\cdot S(a_1)\rtimes S_p(h_{(1,p)})_{(2,p^{-1})})(a_2(h_{(2,p)}\cdot b)\rtimes h_{(3,p)}g)\\
&=(S_p(h_{(1,p)})_{(1,p^{-1})}\cdot S(a_1))(S_p(h_{(1,p)})_{(2,p^{-1})}\cdot (a_2(h_{(2,p)}\cdot b)))\\
&\otimes S_p(h_{(1,p)})_{(3,p^{-1})}(h_{(3,p)}g)\\
&=S_p(h_{(1,p)})_{(1,p^{-1})}\cdot (S(a_1)(a_2(h_{(2,p)}\cdot b)))\otimes S_p(h_{(1,p)})_{(2,p^{-1})}(h_{(3,p)}g)\\
&=\epsilon(a)S_p(h_{(2,p)})\cdot (h_{(3,p)}\cdot b)\otimes (S_p(h_{(1,p)})(h_{(4,p)}g)=\epsilon(a)\epsilon_p(h)b\otimes g,\\
\end{aligned}$$

and
$$\begin{aligned}
&(a\rtimes h)_{(1,p)}(S^*_p((a\rtimes h)_{(2,p)})(b\rtimes g))\\
&=(a_1\rtimes h_{(1,p)})(S^*_p(a_2\rtimes h_{(2,p)})(b\rtimes g))\\
&=(a_1\rtimes h_{(1,p)})((S_p(h_{(3,p)})\cdot S(a_2)\rtimes S_p(h_{(2,p)}))(b\rtimes g))\\
&=a_1(h_{(1,p)}\cdot((S_p(h_{(5,p)})\cdot S(a))(S_p(h_{(4,p)})\cdot b)))\otimes h_{(2,p)}(S_p(h_{(3,p)})g)\\
&=a_1(h_{(1,p)}\cdot ((S_p(h_{(3,p)})\cdot S(a_2))(S_p(h_{(2,p)})\cdot b)))\otimes g\\
&=a_1(h_{(1,p)}\cdot (S_p(h_{(2,p)})\cdot (S(a_2)b)))\otimes g\\
&=\epsilon_p(h)a_1(S(a_2)b)\otimes g=\epsilon(a)\epsilon_p(h)b\otimes g,\\
\end{aligned}$$

the equation (\ref{2c}) is verified.

Similarly, since
$$\begin{aligned}
&((b\rtimes g)(a\rtimes h)_{(1,p)})S^*_p((a\rtimes h)_{(2,p)})\\
&=((b\rtimes g)(a_1\rtimes h_{(1,p)}))S^*_p(a_2\rtimes h_{(2,p)})\\
&=(b(g_{(1,q)}\cdot a_1)\rtimes g_{(2,q)}h_{(1,p)})(S_p(h_{(3,p)})\cdot S(a_2)\rtimes S_p(h_{(2,p)}))\\
&=(b(g_{(1,q)}\cdot a_1))((g_{(2,q)}h_{(1,p)})\cdot(S_p(h_{(4,p)}\cdot S(a_2)))\otimes (g_{(3,q)}h_{(2,p)})S_p(h_{(3,p)})\\
&=(b(g_{(1,q)}\cdot a_1))((g_{(2,q)}h_{(1,p)})\cdot (S_p(h_{(2,p)})\cdot S(a_2)))\otimes g_{(3,q)}\\
&\stackrel{(\ref{6m})}=(b(g_{(1,q)}\cdot a_1))(((g_{(2,q)}h_{(1,p)})S_p(h_{(2,p)}))\cdot S(a_2))\otimes g_{(3,q)}\\
&=\epsilon_p(h)(b(g_{(1,q)}\cdot a_1))(g_{(2,q)}\cdot S(a_2))\otimes g_{(3,q)}\\
&\stackrel{(\ref{6f})}=\epsilon_p(h)(b(g_{(1,q)}\cdot a_1))S(g_{(2,q)}\cdot a_2)\otimes g_{(3,q)}\\
&=\epsilon_p(h)(b(g_{(1,q)}\cdot a)_1)S((g_{(1,q)}\cdot a)_2)\otimes g_{(2,q)}\\
&=\epsilon(a)\epsilon_p(h)b\otimes g,\\
\end{aligned}$$

and
$$\begin{aligned}
&((b\rtimes g)S^*((a\rtimes h)_{(1,p)}))(a\rtimes h)_{(2,p)}\\
&=((b\rtimes g)S^*_p(a_1\rtimes h_{(1,p)}))(a_2\rtimes h_{(2,p)})\\
&=((b\rtimes g)(S_p(h_{(1,p)})_{(1,p^{-1})}\cdot S(a_1)\rtimes S_p(h_{(1,p)})_{(2,p^{-1})}))(a_2\rtimes h_{(2,p)})\\
&=(b(g_{(1,q)}\cdot (S_p(h_{(1,p)})_{(1,p^{-1})}\cdot S(a_1)))\rtimes g_{(2,q)}S_p(h_{(1,p)})_{(2,p^{-1})})(a_2\rtimes h_{(2,p)})\\
&=(b(g_{(1,q)}\cdot (S(h_{(1,p)})_{(1,p^{-1})}\cdot S(a_1)))((g_{(2,q)}S_p(h_{(1,p)})_{(2,p^{-1})})\cdot a_2)\otimes(g_{(3,q)}S_p(h_{(1,p)})_{(3,p^{-1})})h_{(2,p)}\\
&=(b((g_{(1,p)}S_p(h_{(1,p)})_{(1,p^{-1})})_1\cdot a)_1)(g_{(1,q)}S_p(h_{(1,p)})_{(1,p^{-1})})\cdot a)_2\otimes (g_{(2,q)}S_p(h_{(1,p)})_{(2,p^{-1})})h_{(2,p)})\\
&=\epsilon(a)b\otimes (gS_p(h_{(1,p)}))h_{(2,p)}=\epsilon(a)\epsilon(h)b\otimes g,\\
\end{aligned}$$

the equation (\ref{2d}) is verified, therefore $A\otimes H$ is a $Q-$graded Hopf quasigroup.

\end{document}